\documentclass[12pt]{elsarticle}
\usepackage{amssymb}
\usepackage{amsmath}
\usepackage{amsfonts,dsfont}
\makeatletter
\def\ps@pprintTitle{%
 \let\@oddhead\@empty
 \let\@evenhead\@empty
 \def\@oddfoot{}%
 \let\@evenfoot\@oddfoot}
\makeatother

\usepackage{hyperref}

\usepackage{tikz}
\usetikzlibrary{arrows}
\newcommand {\junk}[1]{}
\usepackage{enumerate}

\usepackage{amsthm}

\newtheorem{theorem}{Theorem}[section]
\newtheorem{lemma}[theorem]{Lemma}
\newtheorem{proposition}[theorem]{Proposition}
\newtheorem{corollary}[theorem]{Corollary}
\newtheorem{algorithm}[theorem]{Algorithm}

\theoremstyle{definition}
\newtheorem{definition}[theorem]{Definition}

\newtheorem{example}[theorem]{Example}

\if{

\newtheorem{definition}[theorem]{Definition}

\newtheorem{lemma}[theorem]{Lemma}

\newtheorem{proposition}[theorem]{Lemma}

\newenvironment{proof}[1][Proof]{\noindent\textbf{#1.} }{\ \rule{0.5em}{0.5em}}
}\fi

\def\crit{{\mathcal G}^c}

\def\Nat{{\mathbb N}}

\def\bzero{{\mathbf 0}}

\def\R{{\mathbb R}}

\def\N{{\mathbb N}}
\def\Z{{\mathbb Z}}

\def\Rmax{\R_{\max}}
\def\Rp{\R_+}
\def\Rpnn{\Rp^{n\times n}}

\def\cycle{Z}

\def\digr{{\mathcal G}}

\def\sr{\lambda}
\def\circumf{\operatorname{cr}}

\def\ep{\operatorname{ep}}

\def\wiel{\operatorname{Wi}}
\def\DM{\operatorname{DM}}
\def\sch{\operatorname{Sch}}
\def\kim{\operatorname{Kim}}

\def\rem{\operatorname{rem}}
\def\cyc{\operatorname{c}}
\def\girth{\hat{\operatorname{g}}}

\def\cd{\operatorname{cd}}

\def\harg{\operatorname{HA}}
\def\nacht{\operatorname{N}}
\def\ct{\operatorname{CT}}

\def\thrct{{\mathcal T}^{ct}}

\def\thrha{{\mathcal T}^{ha}}

\def\bnacht{B_{\operatorname{N}}}
\def\bharg{B_{\operatorname{HA}}}
\def\bct{B_{\operatorname{CT}}}

\def\hagr{{\mathcal G}^{ha}}
\def\ctgr{{\mathcal G}^{ct}}

\newcommand{\walkslen}[3]{\mathcal{W}^{#3}(#1\to #2)}

\newcommand{\walkslennode}[4]{\mathcal{W}^{#3}(#1\xrightarrow{#4} #2)}
\newcommand{\walksnode}[3]{\mathcal{W}(#1\xrightarrow{#3} #2)}

\def\subcrit{{\mathcal G}}

\def\0{-\infty}

\begin{document}

\begin{frontmatter}



\setcounter{footnote}{1}
\title{New bounds on the periodicity transient of the powers of a tropical matrix: using cyclicity and factor rank}

\tnotetext[t1]{This work was partially supported by ANR
Perturbations grant (ANR-10-BLAN 0106). 
The work of S. Sergeev was also supported by EPSRC grant EP/P019676/1.}
\author[rvt1]{Arthur Kennedy-Cochran-Patrick}
\ead{axc381@student.bham.ac.uk}

\author[rvt2]{Glenn Merlet}
\ead{glenn.merlet@univ-amu.fr}

\author[rvt3]{Thomas Nowak\fnref{fntn}}
\ead{thomas.nowak@lri.fr}

\author[rvt1]{Serge{\u\i} Sergeev\corref{cor}}
\ead{s.sergeev@bham.ac.uk}

\address[rvt1]{University of Birmingham, School of Mathematics, 
Edgbaston B15 2TT, UK.}
\address[rvt2]{Aix Marseille Univ, CNRS, Centrale Marseille, I2M, Marseille, France}
\address[rvt3]{Universit\'e Paris-Saclay, CNRS, Orsay, France}

\cortext[cor]{Corresponding author. Email: s.sergeev@bham.ac.uk}

\begin{abstract}
Building on the weak CSR approach developed in a previous paper by Merlet, Nowak and Sergeev~\cite{wCSR}, we establish new bounds for the periodicity threshold of the powers of a tropical matrix. According to that approach, bounds on the ultimate 
periodicity threshold take the form of $T=\max(T_1,T_2)$, where $T_1$ is a bound on the time after which 
the weak CSR expansion starts to hold and $T_2$ is a bound on the time after which the first CSR
term starts to dominate. 

The new bounds on $T_1$ and $T_2$ established in this paper make use of the cyclicity of the associated graph
and the (tropical) factor rank of the matrix, which leads to much improved bounds in favorable cases.
For $T_1$, in particular, we obtain new extensions of bounds of Schwarz, Kim and Gregory-Kirkland-Pullman, previously 
known as bounds on exponents of digraphs. For similar bounds on $T_2$, we introduce the novel concept of walk reduction threshold and establish bounds on it that use both cyclicity and factor rank.  
\end{abstract}


\journal{Linear Algebra and its Applications}






\end{frontmatter}

\section{Introduction}

By tropical linear algebra we mean the linear algebra developed over the tropical (max-plus) semiring,
which is the set $\Rmax:=\R\cup\{-\infty\}$ equipped with operations of ``addition'' $\oplus$: 
$a\oplus b:=\max(a,b)$ and ``multiplication'' $\otimes$: $a\otimes b:=a+b$. These operations are naturally 
extended to matrices and vectors in the usual way. In particular, given two matrices 
$A$ and $B$ of appropriate dimensions with entries $a_{ij}$ and $b_{ij}$ in $\Rmax$, we can define $A\oplus B$, which has entries $(A\oplus B)_{ij}:=a_{ij}\oplus b_{ij}$,
and $A\otimes B$, which has entries $(A\otimes B)_{ij}:=\bigoplus_k a_{ik}\otimes b_{kj}$. 

In this paper, we are interested in the 
tropical matrix powers. For a $d\times d$ tropical matrix $A\in\Rmax^{d\times d}$ and general $t\geq 1$ one defines 
$$ A^{t}
=\overbrace{A\otimes A\otimes A\otimes \cdots\otimes  A}^{t\text{ times}}$$
as the $t$\textsuperscript{th} tropical power of $A$. We formally define $A^0=I$, where $I$ is the {\em tropical identity matrix},
with diagonal entries equal to $0$ and off-diagonal entries equal to $-\infty$. 

 It has been known since the work of 
Cohen et al.~\cite{CDQV-83} that the tropical matrix powers, under some mild assumption on $A$, satisfy the following property 
after long enough time $T$ and for some integer $\sigma\geq 1$:
\begin{equation}\label{e:period}
\forall t\geq T\ :\quad A^{t+\sigma} = \lambda^{\otimes \sigma}\otimes A^t.
\end{equation}
Here~$\lambda=\lambda(A)$ is the maximum cycle mean of~$A$. In other words, we have that the tropical matrix powers of 
the matrix $\lambda^-\otimes A$, where $\lambda^-=-\lambda$ is the inverse of $\lambda$ with respect to 
$\otimes=+$, are ultimately periodic with some period $\sigma$.
The smallest~$T$ such that~\eqref{e:period} holds is called the {\em
transient} of~$A$ and denoted by~$T(A)$. 

Many different bounds on $T(A)$ have been formulated and proved since~\eqref{e:period} was observed, see Hartmann and Arguelles~\cite{HA-99}, Bouillard and Gaujal~\cite{BG-00},
Soto y Koelemeijer~\cite{SyK:03}, Akian, Gaubert and Walsh~\cite{AGW-05},  Charron-Bost, F\"{u}gger and Nowak~\cite{CBFN-12}, and  Merlet, Nowak and Sergeev~\cite{wCSR}, for an incomplete list of works on this matter. 

Merlet, Nowak and Sergeev~\cite{wCSR},  in particular, put forward a unifying idea using which most of the previously known bounds could be deduced and improved: the weak CSR expansion. 
This idea stems from a more detailed version of~\eqref{e:period} formulated by Sergeev~\cite{Ser-09}:
there exists a nonnegative integer~$T$ such that
\begin{equation}
\label{e:csr}
\forall t\geq T \ :\quad A^t=(\lambda(A))^{\otimes t}\otimes CS^tR\enspace,
\end{equation}
where the matrices~$C$, $S$, and~$R$ are defined in terms of~$A$ (see~\eqref{csrdef} below) and fulfill
$CS^{t+\sigma}R=CS^tR$ for all~$t\geq 0$. Here and below, the sign $\otimes$ is systematically omitted for the tropical matrix multiplication, but it is kept for the tropical scalar multiplication. Thus the sequence $CS^tR$ is periodic and the smallest~$T$ satisfying~\eqref{e:csr} is~$T(A)$. Note that in an earlier work, considering infinite-dimensional matrices, Akian, Gaubert and Walsh~\cite{AGW-05} gave a similar formulation
originating from the preprints of Cohen~et~al.~\cite{CDQV-83}.

Merlet, Nowak and Sergeev~\cite{wCSR}, based on the earlier results of Sergeev and Schneider~\cite{SerSch}, observed that the tropical matrix powers $A^t$ admit the following weak CSR expansion:
\begin{equation}
\label{weak-exp}
\forall t\geq T\ :\quad A^t=\big((\sr(A))^{\otimes t}\otimes
CS^tR\big)\oplus B^t
\enspace,
\end{equation}
where $C$, $S$, and $R$ are defined in~\eqref{csrdef} below and~$B$
is obtained from~$A$ by setting several entries (typically, all
entries in several rows and columns) to~$\0$. The smallest~$T$, for which~\eqref{weak-exp} holds, 
is denoted it by~$T_1(A,B)$. It is then quite obvious to bound $T(A)\le\max(T_1(A,B),T_2(A,B))$, where $T_2(A,B)$ is the least integer satisfying
\begin{equation}
\label{CSRdominance}
\forall t\geq T\ :\quad (\lambda(A))^{\otimes t}\otimes\big(CS^tR\big)\ge B^t
\enspace.
\end{equation}

Bounds on $T(A)$ then crucially depend on the quality of bounds on $T_1(A,B)$ and $T_2(A,B)$, which 
we are going to improve, with respect to what was obtained in \cite{wCSR}, in this paper. The bounds on 
$T_1(A,B)$ and $T_2(A,B)$ depend also on how matrix $B$ is defined and below we will discuss, following~\cite{wCSR},
three useful schemes for defining $B$: the Nachtigall scheme, the Hartmann-Arguelles scheme and the Cycle Threshold scheme. For the Nachtigall and the Hartmann-Arguelles scheme,  
it was shown in~\cite{wCSR} that  $T_1(A,B)$ is bounded by the {\em Wielandt number}
\begin{equation}
\label{e:wiel}
\wiel(d)=
\begin{cases}
(d-1)^2+1, &\text{if $d>1$},\\
0, & \text{if $d=1$},
\end{cases}
\end{equation}
where $d$ is the dimension of $A$, and by the {\em Dulmage-Mendelsohn number}
\begin{equation}
\label{e:DM}
\DM(\girth,d)=\girth(d-2)+d.
\end{equation}
where $\girth$ is the maximal girth of the strongly connected components of the 
critical graph associated with $A$. These numbers originate from 
seminal works of Wielandt~\cite{Wie-50} and Dulmage and Mendelsohn~\cite{DM-62} on nonnegative matrices. 

One of the main aims in this paper will be to prove that, for the same two schemes, we also 
have the following two bounds on $T_1(A,B)$: {\em Schwarz's bound}
\begin{equation}
\label{e:Sch}
\sch(\gamma,d)=\displaystyle
\gamma \cdot\wiel\left(\left\lfloor\frac{d}{\gamma}\right\rfloor\right)
+ (d\rem \gamma)
\end{equation}
where $\gamma$ is the cyclicity of the digraph associated with $A$ and  
{\em Kim's bound}
\begin{equation} 
\label{e:Kim}
\kim(\gamma,\girth,d)=\displaystyle
\girth\cdot\left(\left\lfloor\frac{d}{\gamma}\right\rfloor-2\right)+ d.
\end{equation}
Here and below, we will denote by~$d\rem \gamma$ the remainder of the Euclidean division of~$d$ by~$\gamma$. These bounds originate from the works of Schwarz~\cite{Sch-70} and Kim~\cite{Kim-79} on binary relations and Boolean matrix powers. For the Nachtigall scheme, the validity of these two bounds is shown in
Theorem~\ref{th:KimSchwarzNacht} (Section~\ref{s:Nacht}), and for the Hartmann-Arguelles scheme in Theorem~\ref{th:HA} (Section~\ref{s:T1AB}). 

Gregory, Kirkland and Pullman~\cite{GKP-95} showed how Boolean rank $r$, also known as Schein rank, can be used to replace $d$ in bounds~\eqref{e:wiel},~\eqref{e:DM},~\eqref{e:Sch} and~\eqref{e:Kim}, with a negligible penalty of adding $1$ to all of these bounds. Following the idea of Merlet et al.~\cite{CritCol} we replace the Boolean rank with the tropical factor rank. Studied by Develin, Santos and Sturmfels in~\cite{DSS-05} as Barvinok rank and further investigated in a number of works (e.g., Akian, Gaubert and Guterman~\cite{AGG-09}), the tropical factor rank of a square matrix over 
$\Rmax$ is a direct generalization of the Boolean or Schein rank. Using the tropical factor rank, we extend the bound of Gregory, Kirkland and Pullman to $T_1(A,B)$ in the case when the Nachtigall scheme is used: see Theorem~\ref{th:rank}.  
 
The proofs in the case of the Nachtigall scheme (Theorems~\ref{th:KimSchwarzNacht}
and~\ref{th:rank}) use the weak CSR expansions and the block decompositions related to cyclic classes. For the Hartmann-Arguelles scheme the same proofs do not work and we need to improve the bounds on the cycle removal threshold ($T_{cr}$) introduced in~\cite{wCSR}. We also introduce a new notion of the walk reduction threshold ($T_{wr}$), which is later used for the bounds on $T_2(A,B)$. New bounds on $T_{cr}$ that make use of the cyclicity of the associated digraph are obtained in Proposition~\ref{p:TcRLinGen}, Corollary~\ref{c:TcRLin} and Proposition~\ref{p:TcRn},  
and some bounds on the walk reduction threshold $T_{wr}$ are offered in Proposition~\ref{p:Twr}, including two bounds that involve the factor rank. 

Some new bounds on $T_2(A,B)$ are also obtained. These bounds make use of the cyclicity of the graph as well as of the tropical factor rank of the matrix and are based on the new bounds on $T_{cr}$ and $T_{wr}$ as well as a result of~\cite{wCSR} relating $T_2(A,B)$ to $T_{cr}$.

In the case when $B$ is defined according to the Cycle Threshold scheme, the connection to $T_{cr}$ established in~\cite{wCSR} does not allow us to obtain true analogues of the classical bounds. However, in this case we also obtain an improved bound on $T_1(A,B)$ in Theorem~\ref{th:T1CT}, involving the cyclicity of the associated digraph. 

\section{Preliminaries}
\label{s:prel}

In this section, we recall the main notions and facts that are necessary to 
understand the main results of the paper and how they are proved. We start by recalling the standard definitions concerning weighted digraphs and walks, and the optimal walk interpretation of the tropical matrix powers.  We also introduce the cyclic classes of a strongly connected digraph and discuss the related block decomposition of the associated matrix. In the last subsection, we formally introduce the CSR decomposition and the three types of weak CSR expansion (Nachtigall, Hatmann-Arguelles and Cycle Threshold), with which we are going to work in the paper.

\subsection{Matrices, graphs and tropical matrix powers}

The {\em digraph associated with a matrix} $\Rmax^{d\times d}$, denoted by $\digr(A)$, 
is the pair $(V,E)$ where $V=\{1,\ldots,d\}$ is the set of nodes and $E\subseteq V\times V$,
where  $(i,j)\in E$ if and only if $a_{ij}\neq -\infty$, is the set of arcs.  Digraph $\digr(A)$ is 
weighted by a function $w\colon E\mapsto \R$, which associates to each arc $(i,j)\in E$ its 
weight $w(i,j)=a_{ij}$. 

A {\em walk} on $\digr(A)=(V,E)$ is a sequence of nodes of $V$ such that every consecutive pair of nodes in this sequence is an arc (i.e., belongs to $E$). A walk in which no node is repeated is called a {\em path}. The weight of a walk $W$, denoted by $p(W)$,
is defined as a tropical product of the weights of all the edges in the walk (i.e.,  the conventional sum of the weights of these edges). Denoting by $\walkslen{i}{j}{t}$ the set of walks, for which the first node is $i$, the  last node is $j$ and the number of edges (i.e., the length) is $t$, one can easily obtain the following {\em optimal path interpretation} of the tropical matrix powers: 
\begin{equation}
\label{e:Atoptpath}
(A^t)_{ij}=\bigoplus\{p(W)\colon W\in \walkslen{i}{j}{t}\}.
\end{equation}
In words, the $(i,j)$th entry of $A^t$ equals to the optimal weight of the walks that connect $i$
to $j$ and have length $t$. When the ends $i$ and $j$ of a walk coincide, it is called a {\em closed walk}, and when there are no repetitions of nodes in the closed walk (except for the coincidence of the ends), such closed walk is called a {\em cycle}. Note that a closed walk can consist of just one node and no arcs: in this case it is an {\em empty walk} of length $0$.

$A$ is called {\em irreducible} if $\digr(A)=(V,E)$ is strongly connected, i.e., if there is a walk connecting $i$ to $j$ for any pair of nodes $i,j\in V$. A digraph $\digr$ is called {\em completely reducible} if it consists of a number of strongly connected components, commonly abbreviated as s.c.c.'s, such that there is no arc connecting any of these components to another.

The assumption that $A$ is irreducible is sufficient for the ultimate periodicity properties~\eqref{e:period} and~\eqref{e:csr} to hold, and the maximum cycle mean $\lambda(A)$, which participates in them, is defined as follows:
\begin{equation}
\begin{split}
\label{e:mcm}
\lambda(A)&=\bigoplus_{k=1}^d\bigoplus_{i_1,\ldots i_k} (a_{i_1i_2}\otimes\ldots\otimes a_{i_ki_1})^{\otimes 1/k}\\
& = \max_{k=1}^d \max\limits_{i_1,\ldots i_k} \frac{a_{i_1i_2}+\ldots+ a_{i_ki_1}}{k}
\end{split}
\end{equation}
On the digraph $\digr(A)$ associated with $A$, this is the maximum arithmetic mean of 
the weight of every cycle (or, equivalently, any closed walk). It is also known to be the unique tropical eigenvalue of 
$A$ when it is irreducible~\cite{But:10}.

The cycles $i_1i_2\ldots i_ki_1$ and nodes within said cycles, on which the maximum cycle mean $\lambda(A)$ is attained, are
called {\em critical}, and the subgraph of $\digr(A)$ consisting of all nodes and arcs belonging to such cycles is called the critical graph of $A$ and is denoted by $\crit(A)$. It is of utmost importance for the description of the long-term behavior of the tropical matrix powers. It is easy to see that $\crit(A)$ is a completely reducible digraph.

For $A\in\Rmax^{d\times d}$ with $\lambda(A)\leq 0$ we can also define the {\em Kleene star} of $A$ as the matrix $A^*$ equal to
\begin{equation*}
A^*=\bigoplus_{i=0}^{+\infty} A^i =\bigoplus_{i=0}^{d-1} A^{i},
\end{equation*}
recalling that $A^0=I$ (the tropical identity matrix).

In the last part of this subsection, let us also define some other digraph parameters
that will be significant for this work. 

For a strongly connected digraph $\digr$, its {\em girth} is defined as the smallest length of a cycle on $\digr$ and denoted by $g(\digr)$. For a completely reducible digraph $\digr$, its {\em max-girth} is defined as the greatest girth of all of its components, and denoted by $\hat{g}(\digr)$.

For a strongly connected digraph $\digr$, its {\em cyclicity} is defined as the greatest common divisor of the lengths of all cycles in $\digr$. For a completely reducible digraph $\digr$, its {\em cyclicity} is defined as the least common multiple of cyclicties of all components of $\digr$, and its {\em max-cyclicity} is defined as the greatest cyclicity of all components of $\digr$. Note that the cyclicity of the critical digraph $\crit(A)$ is commonly taken as the period $\gamma$ in the ultimate periodicity properties~\eqref{e:period} and~\eqref{e:csr} of the tropical matrix powers. Also, by the cyclicity of a matrix $A\in\Rmax^{d\times d}$ we mean the cyclicity of $\digr(A)$.

For a strongly connected digraph $\digr$, its {\em circumference} is defined as the biggest length of a cycle in $\digr$.

The above described parameters and properties of tropical matrix $A\in\Rmax^{d\times d}$ and $\digr(A)$ are stable under tropical diagonal similarity transformations of the form $A\mapsto D^{-}AD$. Here $D^-$ denotes the tropical inverse of $D$: the unique matrix satisfying $D^- D=D D^-=I$ (if it exists). The class of invertible tropical matrices is rather thin: it includes only the {\em finite tropical diagonal matrices} $D$ where all diagonal entries belong to $\R$ and all off-diagonal entries are $-\infty$, tropical permutation matrices $P$  such that, for some permutation $\pi$, their entries $p_{ij}$ are $0$ if $j=\pi(i)$ and $-\infty$ otherwise, and the {\em tropical monomial matrices}, i.e., all products of finite tropical diagonal matrices and tropical permutation matrices. Then, in particular, we have $\sr(A)=\sr(D^-A D)$  for any such transformation, and 
$\crit(A)=\crit(D^- A D)$ when $D$ is a tropical diagonal matrix. Most importantly, we have $(D^-A D)^t=D^- A^t D$ for any such transformation. Such parameters as girth, cyclicity or circumference are also stable under $A\mapsto D^-AD$, as $\digr(D^-AD)$, if we ignore the weights, is obtained from $\digr(A)$ by renumbering the nodes.

\subsection{Cyclic classes and structured irreducible matrices}

In a strongly connected digraph $\digr$, any walk connecting a fixed pair of nodes to each other has the same length modulo cyclicity $\gamma$ of $\digr$, see Brualdi and Ryser~\cite{BR} Lemma 3.4.1. This fact helps to define an equivalence relation on the set of nodes of $\digr$: two nodes $i$ and $j$ belong to the same equivalence class called {\em cyclic class} if the length of any walk between them is a multiple of $\gamma$. Note that Brualdi and Ryser~\cite{BR} call them imprimitivity components, but we will use the terminology that is more common in the tropical algebra.

In connection to this, let $A_1,A_2,\cdots A_\gamma$, for $\gamma$ being arbitrary integer, be matrices with nonnegative entries such that the tropical product $A_i  A_{i+1}$ is well defined for each $i$
(with the indices considered modulo~$\gamma$) and consider
\begin{equation}\label{eq:A}
A:=\left(\begin{array}{ccccc}
 -\infty  &A_1&-\infty&	\cdots&-\infty\\
\vdots&\ddots&\ddots&\ddots&\vdots \\
\vdots&&\ddots&A_{\gamma-2}&-\infty \\
 -\infty  &	\cdots&	\cdots	& -\infty &A_{\gamma-1}\\
A_\gamma&-\infty &\cdots&\cdots&-\infty\\
\end{array}\right).
\end{equation}
Here $-\infty$ mean blocks of matrix $A$ that consist of $-\infty$ entries only.
Notice that any irreducible matrix over $\Rmax$ can be put into this form by permuting the indices, or equivalently, by performing a transformation $A\mapsto P^{-} A P$ where $P$ is a tropical permutation matrix. In this case, we can assume that $\gamma$ is the cyclicity of~$A$, and then the sets of rows $N_i$ on which the submatrices $A_i$ stand (for $i=1,\ldots,\gamma$) are precisely the cyclic classes of $\digr(A)$. The edges connecting $N_i$ to $N_{i+1}$ (modulo $\gamma$) are weighted by entries from $A_i$, and any cycle of $\digr(A)$ contains nodes from each $N_i$ and its length is a multiple of $\gamma$. 

We will be interested in the tropical powers of~$A$ and their limits.
These tropical powers always have a block decomposition compatible with~\eqref{eq:A} and at most one non-$-\infty$ block on each row.
We will denote by~$M_i$ the possibly non-zero block of~$M$ on row~$i$, and in the sequel all indices are always considered modulo~$\gamma$.
This is consistent with~\eqref{eq:A} and for instance, $A^\gamma$ is block diagonal with~$A^\gamma_i=A_i\cdots A_{i+ \gamma-1}$.

We then observe that a cycle with length~$l$ with maximal average weight on~$\digr(A)$
gives a cycle with length~$l/\gamma$ with maximal average weight on each~$\digr(A^\gamma_i)$, and hence we have:
\begin{proposition}\label{p:cyclesAgamma}
 Let $A\in\Rmax^{d\times d}$ be an irreducible matrix of the form~\eqref{eq:A}. 
For each~$i$, we have
 \begin{equation*}
 \begin{split}
  \sr(A^\gamma_i)&=\gamma\cdot  \sr(A),\quad
 \cyc(\digr(A^\gamma_i))=\frac{\cyc(\digr(A))}{\gamma},\\
 \cyc(\crit(A^\gamma_i))&=\frac{\cyc(\crit(A))}{\gamma},\quad
\girth(\crit(A^\gamma_i))=\frac{\girth(\crit(A))}{\gamma}.
 \end{split}
 \end{equation*}
 where $ \sr(B)$ is the maximal average weight of cycles (or circuits) on $\digr(B)$, which is the largest tropical eigenvalue of~$B$,
 $\cyc(\digr)$ is the cyclicity of~$\digr$, and~$\girth(\digr)$ the maximum of the girths (length of shortest cycle)
 of its strongly connected components.
\end{proposition}
The first property listed in this Proposition will be commonly used to
assume that $\sr(A)=0$: in this case, $\sr$ of 
any diagonal block of $A^{\gamma}$ is also equal to $0$.

We also have the following straightforward relations:
\begin{align}
\label{eq:At}
A^{\gamma k+s+t}_i 	&= A^s_{i}(A^\gamma_{i+s})^kA^t_{i+s}\\
\label{eq:Ait}
(A^\gamma_{i})^{k+1}	&=A^{j-i}_i(A^\gamma_{j})^kA^{i-j}_{j}
\end{align}

\subsection{CSR decomposition and weak CSR expansions}
\label{ss:CSR}

In this section we will introduce the concepts related to CSR decomposition and 
weak CSR expansions. While doing this we 
will closely follow the lines of 
\cite{wCSR} and \cite{SerSch}.

For any~$A\in\Rmax^{d\times d}$, let $\sigma$ be the cyclicity of $\crit(A)$  and let
$M=\left((\sr(A)^- \otimes A\big)^{\sigma}\right)^*$ (the Kleene star of $\left((\sr(A)^- \otimes A\big)^{\sigma}\right)$). Define the matrices
$C,S,R\in\Rmax^{d\times d}$ by
\begin{equation}
\label{csrdef}
\begin{split}
c_{ij}&=
\begin{cases}
m_{ij} &\text{if $j$ is in $\crit(A)$}\\
\0 &\text{otherwise,}
\end{cases}\quad
r_{ij}=
\begin{cases}
m_{ij} &\text{if $i$ is in $\crit(A)$}\\
\0 &\text{otherwise,}
\end{cases}\\
s_{ij}&=
\begin{cases}
\sr(A)^- \otimes a_{ij} &\text{if $(i,j)\in \crit(A)$}\\
\0 &\text{otherwise.}
\end{cases}
\end{split}
\end{equation}
Note that for arbitrary $t$, we will often write $CS^tR[A]$ for $CS^tR$ where $C$, $S$ and $R$ are defined 
using $A$. As shown in \cite{SerSch}, the sequence $\{CS^tR[A]\}_{t\geq 1}$ is periodic with period 
$\sigma$ being the cyclicity of $\crit(A)$.

In general, matrix $B$ is defined in terms of the subgraph~$\digr$ of~$\digr(A)$ whose nodes determine the
entries that are set to~$-\infty$ in matrix~$B$:
\begin{equation}\label{e:CSRschemes}
b_{ij} =
\begin{cases}
-\infty & \text{if $i$ or $j$ is a node of $\digr$} \\
a_{ij} & \text{otherwise}.
\end{cases}
\end{equation}

In \cite{wCSR} we introduced three ways of how this $\digr$ and $B$ can be defined. 
Actually, the precise definition of Hartmann-Arguelles scheme and the Cycle Threshold 
scheme given below will not be used further in this paper. However, we are giving it here for the 
sake of completeness and closely following~\cite{wCSR}.

\begin{enumerate}
\item {\em Nachtigall scheme}. This scheme is named after the work of
Nachtigall~\cite{Nacht}. 
In this scheme, the subgraph $\digr$ is the same as $\crit(A)$,
and the matrix~$B$ is defined by~\eqref{e:CSRschemes} using $\digr=\crit(A)$. 
We will denote this $B$ by $\bnacht[A]$.

\item {\em Hartmann-Arguelles scheme}. This scheme is named after the work of 
Hartmann and Arguelles~\cite{HA-99}, and for this scheme we require the notion of max-balanced matrix and max-balancing scaling introduced and studied by 
Hans Schneider and Michael H. Schneider~\cite{SchSch}. 

Let $A\in\Rmax^{d\times d}$ be an associated weight matrix for the graph $\digr(A)$. We say $A$ is \emph{max-balanced} if for any set $W\subseteq \{1,\ldots,d\}$, we have
\begin{equation}
\max_{i\in W , j\notin W} a_{ij} = \max_{i \notin W,j\in W} a_{ij}
\end{equation} 
Equivalently the graph $\digr(A)$ is max-balanced if, for any subset of nodes $W$, the maximal weight over the edges leaving $W$ is equal to the maximal weight over arcs entering $W$. An important property of a max-balanced graph $\digr(A)$ is that for every arc there exists a cycle, on which this arc has the smallest weight~\cite{SchSch}.  

It follows from Schneider and Schneider~\cite{SchSch} that for an irreducible $A\in\Rmax^{d\times d}$ there exists a max-balanced matrix $V=D^- A D$, where $D$ is an appropriate diagonal matrix. $V$ is called a \emph{max-balancing} of $A$.   

Now let $V$ be a max-balancing of $A$. Given $\mu\in\Rmax$, we define the {\em
Hartmann-Arguelles threshold graph} $\thrha(\mu)$ induced by all
arcs~$(i,j)$ in $\digr(A)=\digr(V)$
with $v_{ij}\geq\mu$.
For $\mu=\lambda(A)=\lambda(V)$ we have $\thrha(\mu)=\crit(A)=\crit(V)$.
Let $\mu^{ha}$ be the maximum of $\mu\leq\lambda(A)$ such that $\thrha(\mu)$ has a s.c.c.\
that does not contain any s.c.c.\ of $\crit(A)$.
If no such~$\mu$ exists, then $\mu^{ha}=-\infty$ and $\thrha(\mu^{ha})=\digr(A)$.

The subgraph $\digr=\hagr$ defining~$B$ in the Hartmann-Arguelles scheme is the union
of the s.c.c. of $\thrha(\mu^{ha})$ intersecting~$\crit(A)$.
We denote this matrix~$B$ by $\bharg[A]$.

\item {\em Cycle Threshold scheme}.
For $\mu\in\Rmax$, define the
{\em cycle threshold graph} $\thrct(\mu)$ induced by all nodes
and arcs belonging to the cycles in $\digr(A)$ with mean weight greater or
equal to~$\mu$.
Again, for $\mu=\lambda(A)$ we have
$\thrct(\mu)=\crit(A)$.
Let $\mu^{ct}$ be the maximum of $\mu\leq\lambda(A)$ such that $\thrct(\mu)$ has a s.c.c.\
that does not contain any s.c.c.\ of $\crit(A)$.
If no such~$\mu$ exists, then $\mu^{ct}=\bzero$ and $\thrct(\mu^{ct})=\digr(A)$.

The subgraph $\digr=\ctgr$ defining~$B$ in the cycle threshold scheme is the union
of the s.c.c of~$\thrct(\mu^{ct})$ intersecting~$\crit(A)$.
This matrix~$B$ will be denoted by $\bct[A]$.
\end{enumerate}

When speaking about $T_1(A,B)$ and $T_2(A,B)$ where $B$ is defined using one of the three 
schemes above, we will use the following simplified notation:
\begin{equation*}
\begin{split}
& T_{1,\nacht}(A)=T_1(A,\bnacht[A]),\quad 
T_{1,\harg}(A)=T_1(A,\bharg[A]),\quad
T_{1,\ct}(A)=T_1(A,\bct[A]),\\
& T_{2,\nacht}(A)=T_2(A,\bnacht[A]),\quad 
T_{2,\harg}(A)=T_2(A,\bharg[A]),\quad
T_{2,\ct}(A)=T_2(A,\bct[A]),\\
\end{split}
\end{equation*}
Note that when we perform the operation of tropical scalar multiplication $A'= \mu\otimes A$ with $\mu\in\R$, we have $\sr(A')=\mu\otimes \sr(A)$, 
$\bnacht[A']=\mu\otimes \bnacht[A],$  $\bharg[A']=\mu\otimes \bharg[A]$ and $\bct[A']=\mu\otimes\bct[A]$, while $C$, $S$ and $R$ defined from $A'$ are the same as those defined from $A$. This implies that $T_1(A,B)$ and $T_2(A,B)$ defined for any of these three schemes using~\eqref{weak-exp} and~\eqref{CSRdominance} are invariant under tropical scalar multiplication. The bounds on $T_1(A,B)$ that we will obtain do not depend on the entries of $A$ and therefore are also invariant, and the invariance of new bounds on $T_2(A,B)$ listed in Theorem~\ref{th:T2TcrBounds} is easy to check. This shows that, for any bound on $T_1(A,B)$ and $T_2(A,B)$ proved below, we can assume without loss of generality that $\lambda(A)=0$. That said, some of the statements that we will prove, such as Lemmas~\ref{l:CSR-limit} and~\ref{l:T1N} below, are proved only for $\lambda(A)=0$, which is sufficient for our purposes.

As well as the tropical matrix powers themselves, CSR terms have a well-defined optimal walk 
interpretation~\cite{wCSR,SerSch}. To describe this interpretation, we next recall the following notation for sets of walks
used in~\cite{wCSR}:
\begin{itemize}
\item[1.] Recall that $\walkslen{i}{j}{t}$ is the set of all walks that connect $i$ to $j$ and have length $t$;
\item[2.] For a subgraph $\digr$, let $\walkslennode{i}{j}{t}{\digr}$ be the set of all walks that connect $i$ to $j$, 
have length $t$ and go through a node of $\digr$;
\item[3.] Let $\walkslennode{i}{j}{t,\ell}{\digr}$ be the set of all walks that connect $i$ to $j$, go through
a node of $\digr$ and have length $t$ modulo $\ell$.
\end{itemize}
For a set of walks $\mathcal{W}$ we denote by $p(\mathcal{W})$ the maximal weight of a walk in the set $\mathcal{W}$.

Using this notation we can write the following identities:
\begin{equation}
\label{e:optwalkint}
A^t_{ij}=p(\walkslen{i}{j}{t}),\quad (CS^tR[A])_{ij}=p(\walkslennode{i}{j}{t,\sigma}{\crit(A)}).
\end{equation}
The second of these identities, where $\sigma$ is the cyclicity of $\crit(A)$ and which requires $\lambda(A)=0$, was obtained in~\cite{SerSch}. See also~\cite{wCSR} for an enhanced version of it, which 
we will use below in Proposition~\ref{p:T2Tcr}.

Let us also prove the following facts about $CSR$, which we previously stated in another article~\cite{ReachBnds}, Proposition 2.16. 
Here the limit is understood in terms of the Euclidean topology over $\log(\Rmax)=\Rp$. Obviously, the operations of max-plus algebra are continuous with respect to this topology.

\begin{lemma}
\label{l:CSR-limit}
Let $A\in\Rmax^{d\times d}$ be such that $ \sr(A)=0$. Then for any natural $t$ we have
\begin{itemize}
\item[(i)]   $\lim_{k\to\infty} A^{t+\sigma k}=CS^tR[A]$ for any natural $t$, where $\sigma$ is the cyclicity of 
$\crit(A)$;
\item[(ii)]   $A^rCS^tR[A]=CS^{t+r}R[A]$ for any natural $t$ and $r$.
\end{itemize}
\end{lemma}
\begin{proof}
(i): As established in~\cite{SerSch}, for any $A\in\Rmax^{d\times d}$ with $\sr(A)=0$
we have the CSR expansion 
\begin{equation}
\label{e:CSR-exp}
A^{t+\sigma k}=
CS^tR[A]\oplus \bigoplus_{i=1}^m \lambda_i^{\otimes t+\sigma k}\otimes C_i S_i^{t+\sigma k} R_i, \quad\forall t\geq T 
\end{equation}
for some $m\leq d-1$ and some big enough integer $T$. In this expansion, the sequences $\{C_iS_i^tR_i\}_{t\geq 1}$ 
are periodic (with periods different from $\sigma$), and all $\lambda_i<0$, for $i=1,\ldots,m$. 
It then follows that the tropical sum on the right-hand side of~\eqref{e:CSR-exp} tends to
$-\infty$ as $k\to\infty$, which immediately implies the claim.

(ii): This is now an easy corollary of part (i), since we have
$$
A^r\lim_{k\to\infty} A^{t+\sigma k}=\lim_{k\to\infty} A^{t+r+\sigma k}
$$
by the continuity of tropical arithmetics.
\end{proof}

An alternative way to prove part (i) of the above lemma is to use the weak CSR expansion 
$A^{t+\sigma k}=CS^tR[A]\oplus B^{t+\sigma k}$ where $B$ is defined, e.g., as in the Nachtigall scheme, 
and then use that all cycles in $\digr(B)$ have negative weight, which is equivalent to $\sr(B)<0$.

The following statement holds in the particular case of the Nachtigall expansion:
\begin{lemma}
\label{l:T1N}
Let $A\in\Rmax^{d\times d}$ have $\sr(A)=0$. Then $A^t\geq CS^tR[A]$  if 
and only if $t\geq T_{1,\nacht}(A)$.
\end{lemma}
\begin{proof}
One part of the claim is obvious: if $t\geq T_{1,\nacht}(A)$ then
$A^t=CS^tR[A]\oplus B^t$   implies 
$A^t\geq CS^tR[A]$. 

For the opposite part, let us recall that the set of
walks $\walkslen{i}{j}{t}$ is decomposed into: 1) the set of walks $\walkslennode{i}{j}{t}{\crit(A)}$
that go through a node of $\crit(A)$, 2) the set of walks 
$\walkslen{i}{j}{t}\backslash \walkslennode{i}{j}{t}{\crit(A)}$ that do not go through any node
of $\crit(A)$. By the optimal walk interpretation of tropical matrix powers we have
\begin{equation*}
(A^t)_{ij}=p(\walkslen{i}{j}{t})=p(\walkslennode{i}{j}{t}{\crit(A)})
\oplus p(\walkslen{i}{j}{t}\backslash \walkslennode{i}{j}{t}{\crit(A)}),
\end{equation*}
 where for the first term we have 
\begin{equation*}
p(\walkslennode{i}{j}{t}{\crit(A)})\leq 
p(\walkslennode{i}{j}{t,\sigma}{\crit(A)}=(CS^tR[A])_{ij},
\end{equation*}
using the optimal walk interpretation of CSR terms, and for the second term
\begin{equation*}
 p(\walkslen{i}{j}{t} \backslash \walkslennode{i}{j}{t}{\crit(A)})=B^t_{ij},
\end{equation*}
if $B$ is defined as in the Nachtigall scheme (in particular, $B^t\leq A^t$). 
The above observations imply that we have
\begin{equation}
\label{e:weakCSRNacht}
A^t\leq CS^tR[A]\oplus B^t
\end{equation} 
for any $t$. Now suppose that we have 
$A^t\geq CS^tR[A]$, then we multiply this inequality by $A^r$ for any $r\geq 0$ and obtain $A^{t+r}\geq CS^{t+r}R$ for any $r\geq 0$
using Lemma~\ref{l:CSR-limit} part (ii). As also $A^{t+r}\geq B^{t+r}$ for any $t+r$, we immediately obtain
$A^{t+r}=CS^{t+r}R[A]\oplus B^{t+r}$ for any $r\geq 0$, thus $t\geq T_{1,\nacht}(A)$. 
\end{proof}

\section{The case of Nachtigall expansion}
\label{s:Nacht}

This section is devoted to the case of Nachtigall expansion. Here we prove that the bounds of Schwarz and Kim work for $T_1(A,B)$, when $B$ is defined as in the Nachtigall scheme
(Theorem~\ref{th:KimSchwarzNacht}). We further improve these bounds, as well as bounds of Wielandt and Dulmage-Mendelsohn established in~\cite{wCSR}, for the case when $A$ has a non-trivial tropical factor rank $r$ (Theorem~\ref{th:rank}). The proof technique here makes use mostly of the identities~\eqref{eq:CStR} and~\eqref{eq:CStRAi} and of the block decompositions related to cyclic classes. 

\subsection{The bounds of Kim and Schwarz}
\label{ss:KimSchwarzNacht}

Using Lemma~\ref{l:CSR-limit} we obtain the 
following identities as limits of~\eqref{eq:At} 
and~\eqref{eq:Ait}:
\begin{align}
\label{eq:CStR}
CS^{\gamma k+s+t}R[A]_i	&=A^s_{i}CS^kR[A^\gamma_{i+s}]A^t_{i+s}\\
\label{eq:CStRAi}
CS^{k+1}R[A^\gamma_{i}] &= A^{j-i}_iCS^{k}R[A^\gamma_{j}]A^{i-j}_{j}
\end{align}
in the case when $\sr(A)=0$. 

In all of the proofs in this section we assume without loss of generality that $\sr(A)=0$. By Proposition~\ref{p:cyclesAgamma} we then also have $\sr(A_i^{\gamma})=0$ for all $i$, allowing us to apply Lemmas \ref{l:CSR-limit} and \ref{l:T1N} both to $A$ and to $A_i^{\gamma}$ for any $i$.

\begin{lemma}
\label{l:Aijgamma}
The following two relations hold for all 
$i,j\in\{1,\ldots,\gamma\}$:
\begin{itemize}
\item[(i)] $T_{1,\nacht}(A)\le \gamma \max_i T_{1,\nacht}(A^\gamma_i),$
\item[(ii)]$|T_{1,\nacht}(A^\gamma_i)- T_{1,\nacht}(A^\gamma_j)|\le 1.$
\end{itemize}
\end{lemma}
\begin{proof}

(i) Let $k=\max_i T_{1,\nacht}(A_i^{\gamma})$. We have 
$A_i^{\gamma k}=(A_i^{\gamma})^k$ and its limit 
version $CS^{\gamma k}R[A]_i=CS^kR[A_i^{\gamma}]$.
Using Lemma~\ref{l:T1N} we have
\begin{equation*}
    A_i^{\gamma k}=(A_i^{\gamma})^k\geq CS^kR[A_i^{\gamma}]=CS^{\gamma k}R[A]_i.
\end{equation*}
for all $i$. Using Lemma~\ref{l:T1N} 
again, we obtain $\gamma k\geq T_{1,\nacht}(A)$
thus proving the claim.

(ii): Let 
$k=T_{1,N}(A_j^{\gamma})$. Using~\eqref{eq:Ait},
~\eqref{eq:CStRAi} and Lemma~\ref{l:T1N} we obtain
\begin{equation*}
(A^\gamma_{i})^{k+1}	=A^{j-i}_i(A^\gamma_{j})^kA^{i-j}_{j}
\geq A^{j-i}_iCS^{k}R[A^\gamma_{j}]A^{i-j}_{j}
= CS^{k+1}R[A^\gamma_{i}].
\end{equation*}
Using Lemma~\ref{l:T1N} we obtain $k+1\geq T_{1,\nacht}(A_i^{\gamma})$. This proves the claim since $i$ and $j$ are arbitrary.  
\end{proof}

\if{
From~\eqref{eq:At} and~\eqref{eq:CStR} with~$s=t=0$, we deduce
From~\eqref{eq:Ait} and~\eqref{eq:CStRAi}, we deduce
\begin{equation}\label{eq:TNAivsTNAj}
|T_{1,N}(A^\gamma_i)- T_{1,N}(A^\gamma_j)|\le 1,
\end{equation}
}\fi

The inequalities of Lemma~\ref{l:Aijgamma} immediately imply $T_{1,\nacht}(A)\le \gamma (\wiel(\lfloor\frac{d}{\gamma}\rfloor)+1)$ for a matrix of size~$d$ and cyclicity~$\gamma$. We now improve this result to obtain 
extensions of some bounds for Boolean matrix powers~\cite{GKP-95, Kim-79, Sch-70}.


\begin{theorem}\label{th:KimSchwarzNacht}
Let $A\in\Rmax^{d\times d}$ be irreducible. Denote by $\gamma$ the
cyclicity of $\digr(A)$ and by~$\girth$ the maximal girth of strongly connected components of~$\crit(A)$. The following upper bounds on $T_{1,N}(A)$ hold:
\begin{itemize}
\item[{\rm (i)}] $\sch(\gamma,d) =  \displaystyle
\gamma \cdot\wiel\left(\left\lfloor\frac{d}{\gamma}\right\rfloor\right)
+ (d\rem \gamma) $;
\item[{\rm (ii)}] $\kim(\gamma,\girth,d)= \displaystyle
\girth.\left(\left\lfloor\frac{d}{\gamma}\right\rfloor-2\right)+ d$.
\end{itemize}
 
\end{theorem}

\begin{proof}
The bounds follow from the application of~\eqref{e:wiel} and~\eqref{e:DM}, which are the first two bounds of~\cite[Theorem 4.1]{wCSR}, to the $A^\gamma_{i}$ with minimal size.
This size~$m$ is at most~$\left\lfloor\frac{d}{\gamma}\right\rfloor$.
When it is at most $\left\lfloor\frac{d}{\gamma}\right\rfloor-1$, the bounds follow from the inequalities of Lemma~\ref{l:Aijgamma}.
When $m=\left\lfloor\frac{d}{\gamma}\right\rfloor$, we use the fact that at most $d\rem \gamma$ blocks have a strictly larger size (otherwise the total size would be larger than~$d$).
In this case, we set $$k=\max_{A^\gamma_{i}\textnormal{ has size }m} T_{1,\nacht}(A^\gamma_{i}).$$
Using~\eqref{eq:At} and~\eqref{eq:CStRAi} with 
$k$ as above and $s+t=d\rem\gamma$ and applying Lemma~\ref{l:T1N} we obtain
\begin{equation*}
A^{\gamma k+s+t}_{i}	=A^s_i(A^\gamma_{i+s})^k A^{t}_{i+s}
\geq A^{s}_iCS^{k}R[A^\gamma_{i+s}]A^{t}_{i+s}
= CS^{\gamma k+s+t}R[A]_i.
\end{equation*}
In the above, we select $s$ in such a way that 
$A^{\gamma}_{i+s}$ has size $m=\left\lfloor\frac{d}{\gamma}\right\rfloor$.
Applying Lemma~\ref{l:T1N} again, we obtain
that
$$T_{1,\nacht}(A)\le \gamma \max_{A^\gamma_{i}\textnormal{ has size }m} T_{1,N}(A^\gamma_{i}) +d\rem \gamma.$$
Using Wielandt and Dulmage-Mendelsohn bounds for such blocks together with 
Proposition~\ref{p:cyclesAgamma} we obtain that 
\begin{equation*}
\begin{split}
T_{1,\nacht}(A)&\le \gamma\wiel\left(\left\lfloor\frac{d}{\gamma}\right\rfloor\right)+ d\rem \gamma,\\   
 T_{1,\nacht}(A)&\le \gamma\left( \frac{\girth}{\gamma} \left(\left\lfloor\frac{d}{\gamma}\right\rfloor-2\right)+ \left\lfloor \frac{d}{\gamma}\right\rfloor  \right) +d\rem\gamma\\
 &=\girth\cdot\left(\left\lfloor\frac{d}{\gamma}\right\rfloor-2\right)+ d,     
\end{split}    
\end{equation*}
which concludes the proof.
\end{proof}

\subsection{Using the tropical factor rank}

\label{ss:NachtFactorRank}

Let us first introduce the definition of factor rank and some 
relevant notation.
\begin{definition}
\label{def:frank}
Let $A\in\Rmax^{n\times m}$. The {\em (tropical) factor rank $r$} of $A$ is the smallest $r\in \mathbb{N}$ such that $A=U\otimes L$ where $U\in\Rmax^{n\times r}$ and $L\in\Rmax^{r\times m}$.

We also introduce the following notations:
\begin{equation*}
\check{A}=L\otimes U,\quad 
F=
\begin{pmatrix}
-\infty & U\\
L & -\infty 
\end{pmatrix}
\end{equation*}
\end{definition}

\begin{theorem}\label{th:rank} Let $A\in\Rmax^{d\times d}$ be irreducible. Let~$r$ be the factor rank of~$A$, $\gamma$ the
cyclicity of $\digr(A)$, and $\girth$ the max-girth of~$\crit(A)$. 
The following upper bounds on $T_{1,\nacht}(A)$ hold:
\begin{itemize}
\item[{\rm (i)}] $\displaystyle
\wiel\left(r\right) + 1$;
\item[{\rm (ii)}] $
\DM(\girth,r)+1=\displaystyle\girth \left(r-2\right)+ r +1$.
\item[{\rm (iii)}] $\sch(\gamma,r)+1=\displaystyle
\gamma \wiel\left(\left\lfloor\frac{r}{\gamma}\right\rfloor\right)
+ (r\rem \gamma) + 1$;
\item[{\rm (iv)}] $\kim(\gamma,\girth,r)+1=\displaystyle
\girth \left(\left\lfloor\frac{r}{\gamma}\right\rfloor-2\right)+ r + 1$.
\end{itemize}
The first two bounds apply to reducible matrices as well.
\end{theorem}

\begin{proof}
By definition of the factor rank, we have $A=UL$ for some~$U\in\Rmax^{d\times r}$ and~$L\in\Rmax^{r\times d}.$


For $F$ defined as in Definition~\ref{def:frank} we obtain
$$F^2=\left(\begin{array}{cc}
 A  &-\infty\\
-\infty & \check{A}
\end{array}\right).$$ 
Lemma~\ref{l:Aijgamma} part (ii)
applied to~$F$ with~$\gamma=2$ gives~$T_{1,\nacht}(A)\le T_{1,\nacht}(\check{A})+1$,
and Theorem~\ref{th:rank} follows since the bounds 
$\wiel(r),$ $\DM(\girth,r),$ $\sch(\gamma,r)$ and $\kim(\gamma,\girth,r)$
apply to matrix~$\check{A}=LU$ with size~$r$.
\end{proof}

\section{Cycle Removal and Walk Reduction Threshold}
\label{s:Thresh}

Having obtained the Kim and Schwarz bounds and their factor rank improvements in the case of the Nachtigall scheme, we would like to achieve a similar progress but for the case where $B$ is defined using the Hartmann-Arguelles or the Cycle Threshold scheme. To this aim, in this section we recall the cycle removal threshold ($T_{cr}$) introduced in~\cite{wCSR}. We also introduce a new notion of the walk reduction threshold ($T_{wr}$), which is later used to work with the factor rank, in the new bounds on $T_2(A,B)$. Proposition~\ref{p:TcRLinGen}, Corollary~\ref{c:TcRLin} and Proposition~\ref{p:TcRn} offer new improved bounds on $T_{cr}$
that make use of the cyclicity, and  Proposition~\ref{p:Twr} offers bounds on the walk reduction threshold, and in particular, such bounds that involve the factor rank.

\subsection{Two related notions}
\label{ss:TcrTwr}

In~\cite{wCSR}, we introduced the following definition
\begin{definition}\label{def:Tcr}
Let~$\digr$ be a subgraph of~$\digr(A)$ and~$ \sigma \in\Nat$.

 The {\em cycle removal threshold}~$T_{cr}^ {\sigma}(A,\subcrit)$,
 of~$\subcrit$ is
the smallest nonnegative integer~$T$ for which the following
holds: for each walk~$W\in\walksnode{i}{j}{\subcrit}$ with
length~$\geq T$, there is a walk
$V\in\walksnode{i}{j}{\subcrit}$ obtained from~$W$ by removing
cycles and possibly inserting cycles
of~$\subcrit$ such that $l(V)\equiv_{\sigma} l(W)$ and
$l(V)\le T$.
\end{definition}

However, the following version of the former definition is also natural:

\begin{definition}\label{def:Twr}
Let~$\digr$ be a subgraph of~$\digr(A)$ and~$ \sigma \in\Nat$.

 The {\em walk reduction threshold}~$T_{wr}^ {\sigma} (A,\subcrit)$,
 of~$\subcrit$ is
the smallest nonnegative integer~$T$ for which the following
holds: for each walk~$W\in\walksnode{i}{j}{\subcrit}$ with
length~$\geq T$, there is a walk
$V\in\walksnode{i}{j}{\subcrit}$ such that $p(V)\geq p(W)$, $l(V)\equiv_{\sigma} l(W)$ and
$l(V)\le T$.
\end{definition}

Note that $T_{cr}$ depends only on the unweighted digraph supporting~$\digr(A)$ and is independent of the entries of~$A$,
while~$T_{wr}$ depends on the entries. 

When $\lambda(A)=0$, all closed walks have non positive weights, while all critical closed walks have weight zero. Therefore, removing cycles and possibly inserting cycles as in Definition~\ref{def:Tcr} can only result in a walk with the same or a bigger weight. Thus, we have:

\begin{proposition}
If $\lambda(A)=0$ and $\subcrit$ is a subgraph of~$\crit(A)$, then $T^ {\sigma} _{wr}(A,\subcrit)\le T^ {\sigma} _{cr}(A,\subcrit)$ for any $\sigma\in\Nat$.
\end{proposition}


\subsection{Bounds on the Cycle Removal Threshold}
In this section we prove some new bounds on the Cycle Removal Threshold to be used throughout the paper. One of the starting points is the following bound established in~\cite{wCSR}:

\begin{proposition}[\cite{wCSR}, Proposition 9.5]\label{p:TcRLin}
Let $A\in\Rmax^{d\times d}$ and $\subcrit$ be a subgraph of~$\digr(A)$ with~$d_1$ nodes. Then
$$\forall  \sigma \in\Nat, T_{cr}^ \sigma (A,\subcrit)\le  \sigma  d +d-d_1-1.$$
\end{proposition}

We now develop  this bound for the case when $\digr(A)$ has a nontrivial cyclicity 
$\gamma$.

\begin{proposition}\label{p:TcRLinGen}
Let $A\in\Rmax^{d\times d}$ be irreducible and let
$\subcrit$ be a strongly connected subgraph of~$\digr(A)$. Then
\begin{equation}\label{e:TcrKimGen}
 T_{cr}^{\sigma}(A,\subcrit)\le  \sigma\left\lfloor \frac{d}{\gamma}\right\rfloor +d-\sigma-1,
\end{equation}
where $\gamma$ is the cyclicity of $\digr(A)$ and $\sigma$ is the cyclicity of $\subcrit$.
\end{proposition}

\begin{proof}
 Let $m$ be the size of the smallest cyclic class of~$\digr(A)$.

Let us consider a walk~$W\in\walkslennode{i}{j}{t,\sigma}{\subcrit}$.
If~$W$ does not go through all nodes of~$\subcrit$, then we can insert cycles from $\subcrit$ in it so that the new walk contains all nodes of $\subcrit$ and still belongs to 
$\walkslennode{i}{j}{t,\sigma}{\subcrit}$.

Let~$C_k$ be the first cyclic class of size~$m$ encountered by $W$.
The digraph $\digr(A^{\gamma})$ consists of $\gamma$ isolated strongly connected components, 
whose node sets are the cyclic classes of $\digr(A)$. Denote by $A^{\gamma}_k$ the submatrix 
of $A^{\gamma}$ whose node set is $C_k$.
Let us call $\subcrit_k^{\gamma}$ the digraph which consists 
of all nodes and arcs of $\subcrit^{\gamma}$ that belong to $\digr(A^{\gamma}_k)$.

Then, $W$ can be decomposed into $W=W_1VW_2$ where~$W_1$ has only its end node in~$C_k$ and~$W_2$ only its start node.
By construction, there is a walk~$\tilde{V}$ on~$\digr(A_k^{\gamma})$ with same start and end node as~$V$ and $l(V)=\gamma l(\tilde{V})$. As $W$ goes through all nodes of $\subcrit$, $\tilde{V}$ goes through all (and hence some) nodes of $\subcrit_k^{\gamma}$.

Applying Proposition~\ref{p:TcRLin} to~$\tilde{V}$ and the subgraph $\subcrit_k^{\gamma}$ of~$\digr(A_k^{\gamma})$, we build a walk~$\tilde{V_1}$ with length 
at most~$\frac{\sigma}{\gamma} m +m -d_{1}-1$, where $d_1$ is the number of nodes in $\subcrit_k^{\gamma}$ and 
$l(\tilde{V_1})\equiv_{\frac{\sigma}{\gamma}} l(\tilde{V})$. 
As $d_1\geq l(Z)/\gamma\geq\sigma/\gamma$ where $Z$ is any cycle of $\subcrit$, we also have
$l(\tilde{V_1})\leq \frac{\sigma}{\gamma} m +m -\frac{\sigma}{\gamma}-1$.
This walk can be developed into a walk~$V_2$ on~$\digr(A)$ with length 
at most~$\sigma m +\gamma m -\sigma - \gamma$ and such that 
$l(V_2)\equiv_{\sigma} l(V)$. To bound $l(W_1V_2W_2)$, we consider two cases.

If $m< \left\lfloor \frac{d}{\gamma}\right\rfloor $, we just use that~$l(W_1)\le\gamma-1$ and $l(W_2)\le \gamma-1$ to get
$$l(W_1V_2W_2)\le 2\left(\gamma -1\right) + \left(\gamma +\sigma\right) \left(\left\lfloor \frac{d}{\gamma}\right\rfloor -1\right) -\sigma - \gamma <
\sigma\left\lfloor \frac{d}{\gamma}\right\rfloor -\sigma +d -1.$$

If $m= \left\lfloor \frac{d}{\gamma}\right\rfloor$, we use that~$l(W_2)\le \gamma-1$ and~$l(W_1)\le d \rem \gamma$ to get
$$l(W_1V_2W_2)\le \left(\gamma -1\right) + d \rem \gamma + \left(\gamma +\sigma\right) \left\lfloor \frac{d}{\gamma}\right\rfloor -\sigma - \gamma
= \sigma\left\lfloor \frac{d}{\gamma}\right\rfloor -\sigma +d -1.$$

Thus, we proved~\eqref{e:TcrKimGen}.

\end{proof}

When the subgraph $\subcrit$ is a cycle we obtain the following result:

\begin{corollary}\label{c:TcRLin}
For $A\in\Rmax^{d\times d}$ and $Z$ a cycle of~$\digr(A)$, we have:
\begin{equation}\label{e:TcrKim}
T_{cr}^{l(Z)}(A,\cycle)\le  l(\cycle)\left\lfloor \frac{d}{\gamma}\right\rfloor +d-l(\cycle)-1.
\end{equation}
\end{corollary}

\if{
\begin{proof}
 Let $m$ be the size of the smallest cyclic class of~$\digr(A)$.

Let us consider a walk~$W\in\walkslennode{i}{j}{t,l(Z)}{Z}$.
If~$W$ does not go through all nodes of~$Z$, then we insert a copy of~$Z$ in it.

Let~$C_k$ be the first cyclic class of size~$m$ encountered by $W$.
Let us call $Z_k$ the cycle on~$\digr(A_k^{(\gamma)})$ corresponding to~$Z$ and containing nodes 
from $C_k$. 

Then, $W$ can be decomposed into $W=W_1VW_2$ where~$W_1$ has only end node in~$C_k$ and~$W_2$ only its start node.
By construction, there is a walk~$\tilde{V}$ on~$\digr(A_k^{(\gamma)})$ with same weight, start and end node as~$V$ and $l(V)=\gamma l(\tilde{V})$. As $W$ goes through all nodes of $Z$, $\tilde{V}$ goes through all nodes $Z_k$.

Applying Proposition~\ref{p:TcRLin} to~$\tilde{V}$ and $Z_k$ on~$\digr(A_k^{(\gamma)})$, we build a walk~$\tilde{V_1}$ with length 
at most~$\frac{l}{\gamma} m +m -\frac{l(Z)}{\gamma}-1$ and 
$l(\tilde{V_1})\equiv_{\frac{l}{\gamma}} l(\tilde{V})$, which can be developed into a walk~$V_2$ on~$\digr(A)$ with length 
at most~$l m +\gamma m -l(Z) - \gamma$ and such that 
$l(V_2)\equiv_{l} l(V)$. To bound $l(W_1V_2W_2)$, we consider two cases.

\emph{If $m< \left\lfloor \frac{d}{\gamma}\right\rfloor $}, we just use that~$l(W_1)\le\gamma-1$ and $l(W_2)\le \gamma-1$ to get
$$l(W_1V_2W_2)\le 2\left(\gamma -1\right) + \left(\gamma +l(Z)\right) \left(\left\lfloor \frac{d}{\gamma}\right\rfloor -1\right) -l(Z) - \gamma <
l\left\lfloor \frac{d}{\gamma}\right\rfloor -l(Z) +d -1$$

\emph{If $m= \left\lfloor \frac{d}{\gamma}\right\rfloor$}, we use that~$l(W_2)\le \gamma-1$ and~$l(W_1)\le d \rem \gamma$ to get
$$l(W_1V_2W_2)\le \left(\gamma -1\right) + d \rem \gamma + \left(\gamma +l(Z)\right) \left\lfloor \frac{d}{\gamma}\right\rfloor -l(Z) - \gamma
= l(Z)\left\lfloor \frac{d}{\gamma}\right\rfloor -l(Z) +d -1$$

Thus, we proved~\eqref{e:TcrKim}.

\end{proof}
}\fi

When the cycle of $\digr(A)$  has the maximal possible length,
which is $\gamma\left\lfloor \frac{d}{\gamma}\right\rfloor$,
we also need
\begin{proposition}[\cite{wCSR}]\label{p:TcrHAWielandt}
For $A\in\Rmax^{d\times d}$ and~$\cycle$ a cycle with length~$d$ of~$\digr(A)$, we have $T_{cr}^{d}(A,\cycle)\le d^2-d+1$.
\end{proposition}

Let us improve this bound for the case when $\digr(A)$ has cyclicity $\gamma$.

\begin{proposition}\label{p:TcRn}
 For $A\in\Rmax^{d\times d}$ and~$\cycle$ an elementary cycle with length~$\gamma\left\lfloor \frac{d}{\gamma}\right\rfloor$ of~$\digr(A)$,
we have $T_{cr}^{\gamma \lfloor d/\gamma\rfloor}(A,\cycle)\le \gamma \left(\left\lfloor \frac{d}{\gamma}\right\rfloor-1\right)^2+\gamma+d-1$.
\end{proposition}

\begin{proof}
We first observe that the number of nodes in the smallest cyclic class is
$m=\left\lfloor \frac{d}{\gamma}\right\rfloor$, for otherwise we have 
$m<\left\lfloor \frac{d}{\gamma}\right\rfloor$ and in this case
there is no elementary cycle $Z$
with the length $\gamma\left\lfloor \frac{d}{\gamma}\right\rfloor$.
Indeed, such cycle would have to contain exactly 
$\left\lfloor\frac{d}{\gamma}\right\rfloor$ nodes in each cyclic class, and all these
nodes would have to be different since the cycle is elementary, in contradiction 
with $m<\left\lfloor \frac{d}{\gamma}\right\rfloor$.

So let $m=\left\lfloor \frac{d}{\gamma} \right\rfloor$ be the size of the smallest cyclic class of~$\digr(A)$.

Consider a walk~$W\in\walkslennode{i}{j}{t,l(Z)}{Z}$.
If~$W$ does not go through all nodes of~$Z$, then we insert a copy of~$Z$ in it.

Let~$C_k$ be the first cyclic class of size~$m$ encountered by $W$.
Let us call $Z_k$ the cycle on~$\digr(A_k^{\gamma})$ corresponding to~$Z$ and containing nodes 
from $C_k$. 

We decompose $W$ into $W=W_{1}VW_{2}$ where $W_{1}$ has only an end node in $C_{k}$ and $W_{2}$ has only a start node in $C_{k}$. By construction there is a walk $\tilde{V}$ on $\digr(A^{\gamma})$ with the same start and end node as $V$. Since $W$ contains all nodes of $Z$, walk $\tilde{V}$ contains all nodes of $Z_k$. Then $l(V) = \gamma l(\tilde{V})$. Applying Proposition~\ref{p:TcrHAWielandt} to~$\tilde{V}$ and $Z_k$ on~$\digr(A_k^{\gamma})$, we build a walk~$\tilde{V_1}$ with length 
$l(\tilde{V_1})\leq m^{2} - m +1$ and $l(\tilde{V_1})\equiv_m l(\tilde{V})$, which can be developed into a walk~$V_2$ on~$\digr(A)$ with length 
at most~$\gamma m^{2} -\gamma m + \gamma$ and 
$l(V_2)\equiv_{\gamma m} l(V)$. To bound $l(W_1V_2W_2)$, we can use $W_{1} \le d \rem \gamma$, $W_{2} \le \gamma - 1$, $m =\left\lfloor \frac{d}{\gamma} \right\rfloor$,
and $d=\gamma\left\lfloor\frac{d}{\gamma}\right\rfloor+ (d\rem\gamma)$ to obtain
\begin{align*}
l(W_{1}V_{2}W_{2}) & \le (\gamma - 1) + d \rem \gamma + \gamma \left\lfloor \frac{d}{\gamma} \right\rfloor^{2} - \gamma \left\lfloor \frac{d}{\gamma} \right\rfloor +\gamma \\
& = \gamma-1+d+\gamma
\left(\left\lfloor \frac{d}{\gamma} \right\rfloor-1\right)^2.
\end{align*} 

Thus, we proved the claim.
\end{proof}

\subsection{Bounds on the Walk Reduction Threshold}

Now we will obtain some bounds on the walk reduction threshold that involve the factor rank~$r$. The following elementary
number-theoretic lemma will be especially useful in what follows.
Its origins were briefly discussed by
Aigner and Ziegler~\cite{AZ:01}, p.~133. In the context of
tropical matrix powers, it was introduced by
Hartmann and Arguelles~\cite{HA-99}.
\begin{lemma}
\label{l:HA}
Let $a_1,\ldots, a_{s}\in\Z$. Then there is a nonempty subset 
$I\subseteq\{1,\ldots,s\}$ with 
$\sum_{i\in I} a_i\equiv_s 0$.
\end{lemma}

In what follows, it will be also quite important to lift a walk in $\digr(A)$ to a walk in $\digr(F)$ and then to pass from cycles and s.c.c. of $\crit(A)$ to the related cycles and s.c.c. of $\crit(\check{A})$ (where $F$
and $\check{A}$ are defined in Section~\ref{ss:NachtFactorRank}). We will now define the lift and the relation more formally.

Let $W=i_1\ldots i_m$ be a walk on $\digr(A)$ or on $\digr(\check{A})$. Then walk $\tilde{W}$ is called a {\em lift} of $W$ to $\digr(F)$ if $\Tilde{W}=i_1j_1\ldots j_{m-1}i_m$ is a walk on $\digr(F)$ such that $p(\Tilde{W})=p(W)$.

Let $W=i_1\ldots i_m$ with $i_1=i_m$ be a closed walk on $\digr(A)$ (resp. on $\digr(\check{A})$). 
Then a closed walk $V=j_1\ldots j_m$ with $j_1=j_m$ on $\digr(\check{A})$ (resp. on $\digr(A)$) is called {\em related} to $W$ if   $p(W)=p(V)$ and $\Tilde{W}=i_1j_1\ldots j_{m-1}i_m$ is a lift of~$W$.

Let $\digr$ be a completely reducible subgraph of $\digr(A)$ (resp. of $\digr(\check{A})$). Then, subgraph $\check\digr$ of $\digr(\check{A})$ 
(resp. of $\digr(A)$) is called   
{\em related} to $\digr$ if it consists of all nodes and arcs of all closed walks that are related to the closed walks of $\digr$.

Any closed walk related to a closed walk on $\crit(A)$ is a walk on $\crit(\check{A})$, and vice versa. Moreover, we can make the following observations, assuming without loss of generality that $\lambda(F)=\lambda(A)=\lambda(\check{A})=0$.

\begin{lemma}
\label{l:girth-related}
Let $\digr$ be a cycle $\cycle$ that has the smallest length among the cycles of the s.c.c. of $\crit(A)$ to which it belongs. Then any closed walk $\check\cycle$ related to $\cycle$ is also a cycle of the same length.
\end{lemma}
\begin{proof}
It is obvious that if $\check\cycle$ is a cycle then it has the same length as $\cycle$. Therefore, assume that $\check\cycle$ is not a cycle, in which case part of it is a critical cycle $\check{Y}$ of $\digr(\check{A})$, which is related to a critical closed walk $Y$ of $\digr(A)$ that goes through some nodes of $Z$ and has a smaller length than $Z$, a contradiction.
\end{proof}

The following statement was implicit in \cite{CritCol}.

\begin{lemma}
\label{l:crit-related}
If $\crit(A)$ has s.c.c. $\subcrit_1,\ldots,\subcrit_m$ then $\crit(\check{A})$ has s.c.c. $\check\subcrit_1,\ldots\check\subcrit_m$, which can be numbered so that $\subcrit_i$ and $\check\subcrit_i$ are related to each other for each $i=1,\ldots,m$. Also, 
$\subcrit_i$ and $\check\subcrit_i$ have the same girth and cyclicity. 
\end{lemma}
\begin{proof}
Let $\subcrit$ be a s.c.c. of $\crit(A)$.
For each closed walk in $\subcrit$, consider a common lift of this walk and its related walk in $\check\subcrit$ to $\digr(F)$. Taking all nodes and arcs of such lifts we obtain a s.c.c. of $\crit(F)$: it is obviously strongly connected, and possibility to add new arcs or new nodes would contradict the maximality of $\subcrit$ or the definition of $\check\subcrit$. The connectivity and maximality of this s.c.c. of $\crit(F)$ implies the same for $\check\subcrit$, thus it is also an s.c.c. of $\crit(\check{A})$. The first part of the statement is now obvious.

 The cyclicities of  $\subcrit_i$ and $\check\subcrit_i$ are equal to the g.c.d.'s of the lengths of closed walks in them. As the related closed walks have the same length and $\subcrit_i$ and $\check\subcrit_i$ are related to each other, these cyclicities are equal. 
 
 The equality between girths follows from Lemma~\ref{l:girth-related}.  
\end{proof}



We will now prove the following bounds~:
\begin{proposition}\label{p:Twr}
 Let $A\in\Rmax^{d\times d}$ with~$\sr(A)=0$ have factor rank~$r$ and let~$\subcrit$ be a strongly connected subgraph of~$\crit(A)$
 whose related subgraph $\check{\subcrit}$ in $\digr(\check{A})$ has $|\check{\subcrit}|$ nodes and circumference~$\circumf(\check{\subcrit})$.
 Then
 \begin{itemize} 
     \item[{\rm (i)}]  $T_{wr}^{l}(A,\subcrit)\le 1+ r\left( l+1\right)-|\check\subcrit|$ for any $l\in\Nat$;
     \item[{\rm (ii)}]  If $A$ is irreducible with cyclicity~$\gamma$ and $l\in\gamma\N$, then
     
     $T_{wr}^{l}(A,\subcrit)\le  l\left\lfloor\frac{r}{\gamma}\right\rfloor +r -\circumf(\check\subcrit)+\gamma$;
    \item[{\rm (iii)}]
     If the factor rank $r$ is equal to the max-girth of~$\crit(A)$ and $\cycle$ is a critical cycle with length~$r$, then $T_{wr}^{r}(A,\cycle)\le \wiel(r)+r+1$;
    \item[{\rm (iv)}]
    Let  $A$ be irreducible with cyclicity~$\gamma$, and let  $\gamma\left\lfloor\frac{r}{\gamma}\right\rfloor$ be the max-girth of $\crit(A)$.
   If $\cycle$ is a critical cycle with length~$\gamma\left\lfloor\frac{r}{\gamma}\right\rfloor$ that attains the max-girth in an s.c.c. of $\crit(A)$, then
    $T_{wr}^{\gamma\left\lfloor\frac{r}{\gamma}\right\rfloor  }(A,\cycle)\le \gamma\wiel\left(\left\lfloor\frac{r}{\gamma}\right\rfloor\right)+r+\gamma$.
    
 \end{itemize}
\end{proposition}

\begin{proof}
 The general idea is always the same: we start with a walk~$W \in \walkslennode{i}{j}{t,l}{\subcrit}$ on~$\digr(A)$,
 to which we associate a walk~$\check{W}$ on~$\digr(\check{A})$ and
 we insert and remove cycles from this walk in a way that ensures that when we come back to a walk~$V$ on~$\digr(A)$, the walk still belongs to~$\walkslennode{i}{j}{t,l}{\subcrit}$.
 To get better statements, we actually work with a walk on $\digr(F)$ rather than $\digr(\check{A})$
 but we essentially consider the nodes of the walk that belong to $\digr(\check{A})$.
 Let us go into the details of the different cases.
 
 {\bf Proof of (i):}  Walk $W$ is lifted to a walk~$\tilde{W}$ on $\digr(F)$ with the same weight, whose second and penultimate node belong to~$\digr(\check{A})$.
  We identify the first occurrence~$o_1$ of a node of~$\check\subcrit$ if there is one and insert $l$ copies of the lift to~$\digr(F)$ of a cycle of~$\subcrit$ there.
  If there is no node of~$\check\subcrit$, we insert $l$ copies of a lift to~$\digr(F)$ of a cycle of $\subcrit$ at the first occurrence of a node of~$\subcrit$,
  or two copies of it if the original cycle of $\subcrit$ has length~$1$ and~$l=1$.
  In all cases, it ensures that there is a node~$m$ of~$\subcrit$ between the first two occurrences~$o_1$ and~$o_2$ of nodes of~$\check\subcrit$.
  Now, for each node~$k$ of~$\digr(\check{A})$ that is not in $\check\subcrit$, we count the total number of occurrences of~$k$ in $\tilde{W}$, which can be found either on the left of~$o_1$ or on the right of~$o_2$.
  If $k$ occurs more than $l+1$ times in total, then we can remove arcs between occurrences of~$k$
  and preserve the arcs $o_1\rightarrow m\rightarrow o_2$ and the length of~$\tilde{W}$ modulo~$2l$, by applying Lemma~\ref{l:HA}
  to the set of all lengths of the subwalks between the consecutive occurrences of~$k$ excluding the subwalk between the last occurrence before $o_1$ and the first occurrence after $o_2$.
  If $k$~belongs to~$\check\subcrit$, then $o_1$ is the first position where it can occur. Hence, if $k$ does not occur in~$o_1$, then
  all occurrences are on the right of~$o_2$ (possibly including $o_2$), so we apply Lemma~\ref{l:HA} to the set of the lengths of all subwalks between consecutive occurrences of $k$. We can remove subwalks as soon as there are more than~$l$ occurrences.
  The walk we obtain contains the nodes of~$\check\subcrit$ at most $l$ times outside position~$o_1$ and the other $r-|\check\subcrit|$ nodes of~$\digr(\check{A})$ at most $l+1$ times. The total number of nodes of~$\digr(\check{A})$ is at most $(l+1)r-|\check\subcrit|+1$ ,
  thus the length of the walk is at most $1+2((l+1)r-|\check\subcrit|)+1=2((l+1)r-|\check\subcrit|+1)$
  and the walk~$V$ we get by keeping only nodes of~$\digr(A)$ has length at most $(l+1)r-|\check\subcrit|+1$ and the same length as~$W$ modulo~$l$.
  Moreover, it contains the node~$m$ of~$\subcrit$ and has weight at least $p(\tilde{W})=p(W)$.
  
 {\bf Proof of (ii):} Walk $W$ is lifted to a walk~$\tilde{W}$ on $\digr(F)$ with the same weight, whose second and penultimate node belong to~$\digr(\check{A})$.
 As $W$ traverses a node of $\subcrit$, we can insert $l$ copies of the lift to~$\digr(F)$ of a cycle of~$\subcrit$ at such node into $\tilde{W}$. This ensures 
 that $\tilde{W}$ contains a node of~$\check\subcrit$ in each cyclic class of~$\digr(\check{A})$.
  We identify the first occurrence~$o_1$ of a node of~$\check\subcrit$ in one of the smallest cyclic classes of ~$\digr(\check{A})$.
  Let us denote by~${\check E}$ the cyclic class of the node in~$o_1$ and insert $l$ copies of a lift to~$\digr(F)$ of a cycle of~$\subcrit$ in~$o_1$
  to ensure that there is a node~$m$ of~$\subcrit$ between the first two occurrence~$o_1$ and~$o_2$ of nodes of~${\check E}$ (or two copies of that lift if 
$l=1$ and the length of the original cycle is $1$).
  Each cyclic class of $\digr(\check{A})$, in particular $\check E$, occurs every~$2\gamma$ nodes of the walk.
  Now, for each node~$k$ of~${\check E}$, we count the total number of occurrences of~$k$ remaining after removing subwalks between consecutive occurrences of $k$ by means of Lemma~\ref{l:HA} while maintaining the subwalk between $o_1$ and $o_2$ intact.  
  Similarly to (i), after reduction there are at most $(\frac{l}{\gamma}+1)|{\check E}|-|{\check E}\cap\check\subcrit|+1$
  occurrences of nodes of~${\check E}$.
  We know that~$|{\check E}|\le \left\lfloor\frac{r}{\gamma}\right\rfloor$ and consider two cases.
  
  \textbf{ Case a) : } $|{\check E}|= \left\lfloor\frac{r}{\gamma}\right\rfloor$
  
  In this case, there are at most $r \rem\gamma$ classes
  with more than $\left\lfloor\frac{r}{\gamma}\right\rfloor$ nodes, so we get a walk of length at most $1+2(r \rem\gamma)$
  before reaching~${\check E}$ from the starting node.
  Then, we have a walk of length at most $2\gamma \left((\frac{l}{\gamma}+1)|{\check E}|-|{\check E}\cap\check\subcrit|+1\right)$
  between the first and last node of~${\check E}$.  After leaving~${\check E}$ for the last time
  it goes through at most~$\gamma-1$ cyclic classes of~$\digr(\check A)$ and thus has length at most $2(\gamma -1)+1$.
  
  Finally, the walk has length at most
  $1+2(r \rem\gamma) +2\gamma\left( ( \frac{l}{\gamma}  +1)\left\lfloor\frac{r}{\gamma}\right\rfloor -|{\check E}\cap\check\subcrit| \right) +2(\gamma -1)+1
  = 2\left(l \left\lfloor\frac{r}{\gamma}\right\rfloor +r - \gamma|{\check E}\cap\check\subcrit| +\gamma \right)$.
  Since $|{\check E}\cap\check\subcrit|\ge \frac{\circumf(\check\subcrit)}{\gamma}$, we get the desired bound.
  
  \textbf{ Case b) : }  $|{\check E}|< \left\lfloor\frac{r}{\gamma}\right\rfloor$.
  
  Before reaching~${\check E}$ from the starting node, the walk goes through at most~$\gamma-1$ cyclic classes of~$\digr(\check A)$ and thus has length at most $2(\gamma -1)+1$.
  Then, we have a walk of length at most $2\gamma \left((\frac{l}{\gamma}+1)|{\check E}|-|{\check E}\cap\check\subcrit|+1\right)$
  between the first and last node of~${\check E}$.  After leaving~${\check E}$ for the last time it  has length at most $2(\gamma -1)+1$ as in case~\textbf{a)}.
  Finally, the walk has length at most
  $2(\gamma -1)+1 + 2\gamma\left(( \frac{l}{\gamma}  +1)\left(\left\lfloor\frac{r}{\gamma}\right\rfloor -1\right) -|{\check E}\cap\check\subcrit| \right) +2(\gamma -1)+1$, 
  which is stricly less than the bound in case a).
  
{\bf Proof of (iii):}
Let $\check\cycle$ be a closed walk of $\crit(\check{A})$ that is related to $\cycle$.
In this case, $\check\cycle$ is a cycle of length~$r$by Lemma~\ref{l:girth-related}, thus all nodes of~$\digr(\check A)$ belong to~$\check\cycle$.
  
  $W\in\walksnode{i}{j}{\cycle}$ is lifted to a walk~$\tilde{W}$ on  $\digr(F)$ with the same weight,
  whose second and penultimate nodes~$i_1$ and~$j_1$ belong to~$\digr(\check{A})$ and thus to~$\check\cycle$.
  We insert in~$i_1$ a lift to~$\digr(F)$ of~$\cycle$ to ensure that the walk starts with $ii_1i_2i_3$ where $i_2$ belongs to~$\cycle$ and
  $i_3$ to $\digr(\check{A})$.
  The subwalk between $i_3$ and~$j_1$ defines a walk on~$\digr(\check A)$, to which we apply Proposition~\ref{p:TcrHAWielandt} to get a new walk
  with length at most $r^2 -r +1$.
  Lifting again to $\digr(F)$ and adding $ii_1i_2i_3$ at start and~$j$ as end node, we have a new walk on $\digr(F)$ which defines a walk in
  $\walksnode{i}{j}{\cycle}$ with length at most~$r^2 -r +3=\wiel(r)+r+1$.
  
  {\bf Proof of (iv):} In this case, $\check\cycle$ is also a critical cycle of length~$\gamma \left\lfloor\frac{r}{\gamma}\right\rfloor$ by Lemma~\ref{l:girth-related}, so
   each cyclic class of~$\digr(\check A)$ contains exactly $\left\lfloor\frac{r}{\gamma}\right\rfloor$ nodes of~$\check\cycle$.
   Therefore, there are at most~$r\rem\gamma$ of $\digr(\check A)$ classes that contain other nodes.
   $W\in\walksnode{i}{j}{\cycle}$ is lifted to a walk~$\tilde{W}$ to $\digr(F)$ with the same weight,
  whose second and penultimate nodes~$i_1$ and~$j_1$ belong to~$\digr(\check{A})$.
  This walk reaches a cyclic class~$\check E$ of~$\digr(\check{A})$ consisting only of $\left\lfloor\frac{r}{\gamma}\right\rfloor$ nodes
  of~$\check\cycle$ after at most $r\rem\gamma$ nodes of~$\digr(\check{A})$.
  We identify the first occurrence~$o_1$ of a node of~$\check E$ and insert a lift to~$\digr(F)$ of~$\cycle$ there, so that all nodes of~$\digr(A)$
  between~$o_1$ and the second occurrence~$o_2$ of a node of~$\check E$ belong to~$Z$. Let~$o_3$ be the last occurrence of a node of~$\check E$.
  We split our walk into~$W_1$ from~$i$ to~$o_2$, $W_2$ from~$o_2$ to~$o_3$ and $W_3$ from~$o_3$ to~$j$.
  $W_2$ defines a walk on $\digr(\check{A}^\gamma)$ whose nodes all belong to~$\check E$, to which we apply Proposition~\ref{p:TcrHAWielandt}.
  This gives a walk of length at most~$\left\lfloor\frac{r}{\gamma}\right\rfloor^2-\left\lfloor\frac{r}{\gamma}\right\rfloor+1
  =\wiel\left(\left\lfloor\frac{r}{\gamma}\right\rfloor\right)+\left\lfloor\frac{r}{\gamma}\right\rfloor-1$,
  which we lift again to~$W_4$ on~$\digr(F)$. As in case~\textbf{(iii)}, we have $l(W_1)\leq 1+2r\rem\gamma +2\gamma$ and $l(W_3)\leq 2\gamma -1$, so that the length of $W_1W_4W_3$ is at most
$$
1+2r\rem\gamma+2\gamma +2\gamma \left( \wiel\left(\left\lfloor\frac{r}{\gamma}\right\rfloor\right)+\left\lfloor\frac{r}{\gamma}\right\rfloor-1\right)+2\gamma -1,
$$
and this walk traverses a node of $\cycle$. Contracting it back to a walk on $\digr(A)$, we obtain a walk, the length of which is 
bounded by $\gamma\wiel\left(\left\lfloor\frac{r}{\gamma}\right\rfloor\right)+r+\gamma$ and which also traverses a node of $\cycle$.


\end{proof}

\section{Bounds for $T_1(A,B)$ using
Cycle Removal Threshold}

\label{s:T1AB}

Here we first deduce the bounds if Schwarz and Kim in the case when $B$ is defined according to the Hartmann-Arguelles scheme (Theorem~\ref{th:HA}), using the new bounds on $T_{cr}$ obtained in Section~\ref{s:Thresh}. 
The second subsection achieves a result on $T_1(A,B)$ in the case of Cycle Threshold expansion using the cyclicity (Theorem~\ref{th:T1CT}).

\subsection{The case of Hartmann and Arguelles expansion}

Let us first recall the following link between the cycle
removal threshold and $T_1(A,B)$. The statement will
require the following notions, which we now introduce.

Let~$\digr$ be a subgraph of~$\digr(A)$ and~$\sigma\in\Nat$.

The {\em exploration penalty}~$\ep^\sigma(i)$ of a node~$i\in\digr$
is the least~$T\in\Nat$ such that for any multiple~$t$
of~$\sigma$ greater or equal to~$T$, there is a closed walk
on~$\digr$ with length~$t$ starting at~$i$.

The {\em exploration penalty}~$\ep^\sigma(\digr)$
of $\digr$ is the maximum of
the~$\ep^\sigma(i)$ for~\mbox{$i\in\digr$.}

A subgraph $\subcrit$ of $\crit(A)$ is called {\em representing}, if it is completely reducible and each s.c.c. of $\crit(A)$ contains exactly one s.c.c. of $\subcrit$.

We will use the following bound for 
$T_1(A,B)$:
\begin{equation}
\label{e:t1B-ineq}
\begin{split}
T_1(A,B)
&\le \max_{l=1,\ldots,m}
\left(T_{cr}^{\sigma_l}(A,\subcrit_l)-\sigma_l+1+\ep^{\sigma_l}(\subcrit_l)\right)
\end{split}
\end{equation}
Here $\subcrit_1,\cdots,\subcrit_m$ are the s.c.c.'s of a representing subgraph $\subcrit$ of~$\crit(A)$
and $\sigma_l$ are the cyclicities of $\subcrit_l$
for $l\in \{1,\ldots,m\}$.

\begin{proposition}[\cite{wCSR}, Proposition 6.5]
\label{p:TcrToT1}
Bound~\eqref{e:t1B-ineq} holds when
$B=B_{\nacht}(A)$ or $B=B_{\harg}(A)$.
\end{proposition}

We note that in the case of Nachtigall 
expansion and denoting $\Tilde{A}=A-\lambda(A)$ we can replace $T_{cr}$ by $T_{wr}$ in~\eqref{e:t1B-ineq}:
\begin{equation}
\label{e:t1B-ineqMod}
T_1(A,B)
\le \max_{l=1,\ldots,m}
\left(T_{wr}^{\sigma_l}(\tilde{A},\subcrit_l)-\sigma_l+1+\ep^{\sigma_l}(\subcrit_l)\right).
\end{equation}
We are not going to use this observation here, although it gives an alternative way to derive bounds that use factor rank for the case of Nachtigall expansion. Unfortunately, in the case of Hartmann-Arguelles expansion, the proof of \cite{wCSR}[Proposition 6.5] does not allow for such a replacement.

Proposition~\ref{p:TcrToT1} asserts that~\eqref{e:t1B-ineq}
holds not only for the Nachtigall but also for the Hartmann-Arguelles version of the weak CSR expansion. 
Note also that Lemma~\ref{l:T1N} does not hold 
in the case of the Hartmann-Arguelles expansion.

Bound~\eqref{e:t1B-ineq} will be used only with $\subcrit_l$ being cycles. In this case~$\sigma_l=l(\subcrit_l)$ and $\ep^{\sigma_l}(\subcrit_l)=0$
for each $l=1,\ldots,m:$ the closed walks of lengths $0,\sigma_l,2\sigma_l,\ldots$ are the empty walk, cycle $\subcrit_l$ and walks consisting of repetitions of $\subcrit_l$. 

On our way to Kim and Schwarz bounds for $T_{1,\harg}(A)$ that make use of the cyclicity $\gamma$, let us first pay attention to the case $\left\lfloor\frac{d}{\gamma}\right\rfloor=1,$
for which we will not use Proposition~\ref{p:TcrToT1}.

\begin{proposition}\label{p:n/gamma=1}
If $d<2\gamma$ then for any $A\in\Rmax^{d\times d}$ with cyclicity $\gamma$ such that $\sr(A)\neq \0$, and any~$t\ge d\rem \gamma$, we have~$A^t=(\lambda(A))^{\otimes t}\otimes CS^tR$.
\end{proposition}

\begin{proof}
Without loss of generality, we assume that $\sr(A)=0$.

Let us first notice that all cycles of~$\digr(A)$ have length~$\gamma$, since their length is less than $2\gamma$ and divisible by~$\gamma$.
In particular, $\crit(A)$ has cyclicity~$\gamma$ and all cycles of $\crit(A)$ have length $\gamma$.
Moreover, at most~$d\rem \gamma$ cyclic classes of~$\digr(A)$ have more than one node,
so that there is a class with only one node, $\crit(A)$ is strongly connected, and the nodes in such one-node classes are critical.

{\bf Proof of $(CS^tR)_{ij}\leq A^t_{ij}$.}\\
Let us take an optimal walk $W\in\walkslennode{i}{j}{t,\gamma}{\crit(A)}$, i.e. such 
that $p(W)=(CS^tR)_{ij}$, with minimal length among walks of this type.

Let us show that 
\begin{equation}\label{eq:lWleqt}
 l(W)\leq t.
\end{equation}

If $l(W)\leq d\rem\gamma$, then~\eqref{eq:lWleqt} is obvious. Otherwise, there is a path $P_1$ which is a prefix of $W$ and which connects $i$ to the first occurrence of the only node $k$ of a cyclic class with $1$ element. Denote by $P_2$ the suffix of $W$ which connects the last occurrence of $k$ to $j$. Thus we obtain a decomposition $
W=P_1UP_2$, where $U$ is a closed walk and thus can be decomposed into cycles. Note that all these cycles have length $\gamma$ by above arguments, and that all of them 
are critical, or this contradicts the optimality of $p(W)$. But then $p(P_1P_2)=p(W)$ and, by the minimality of $l(W)$, $W=P_1P_2$.

We have $l(P_1)\leq d\rem\gamma$ and $l(P_2)\leq\gamma -1$ (for otherwise $P_2$ would contain $k$ at least twice, contradicting its definition), 
hence $l(W)<\gamma+d\rem\gamma\leq t+\gamma$.
But $l(W)\equiv_\gamma t$, so that \eqref{eq:lWleqt} is proved.

By its definition, $W$ traverses a critical node and $t\equiv_{\gamma} l(W)$. Since it also satisfies~\eqref{eq:lWleqt}, we can form a walk~$V$ of 
length $t$ by possibly inserting a number of critical cycles into $W$ (recall that all of them have length $\gamma$). By doing so, we show that $p(W)=p(V)\leq (A^t)_{ij}$. 

{\bf Proof of $A^t_{ij}\le (CS^tR)_{ij}$.}\\
Since at most~$d\rem \gamma$ cyclic classes of~$\digr(A)$ have more than one node,
all walks on~$\digr(A)$ with length~$t\ge d \rem \gamma $ meet a cyclic class with only one node, which is critical. Hence $A^t_{ij}\le (CS^tR)_{ij}$, using the optimal walk interpretation~\eqref{e:optwalkint}.
\end{proof}

We now can prove the main result of this section.

\begin{theorem}\label{th:HA}
The bounds of Theorem~\ref{th:KimSchwarzNacht} 
apply to~$T_1(A,B)$ whenever we have~\eqref{e:t1B-ineq}, and in particular for 
$B=B_{\harg}$.
\end{theorem}

\begin{proof}
Let $\subcrit_1,\ldots,\subcrit_m$ be the s.c.c. of 
$\crit(A)$ and let $Z_1,\ldots,Z_m$ be the cycles of
minimal length in those components.
Using Corollary~\ref{c:TcRLin} with $Z=Z_k$
for any $k\in\{1,\ldots,m\}$, we have
\begin{align*}
    T_{cr}^{l(Z_k)} (A,Z_k) & \leq l(Z_k) \left\lfloor\frac{d}{\gamma}\right\rfloor + d - l(Z_k) -1,\\
    T_{cr}^{l(Z_k)}(A,Z_k) -l(Z_k) +1 & \leq l(Z_k)  \left( \left\lfloor\frac{d}{\gamma}\right\rfloor -2 \right) + d
\end{align*}
for all $k\in\{1,\ldots,m\}.$
Combining it with Proposition~\ref{p:TcrToT1}, we write
\begin{align*}
    T_1(A,B) & \leq 
    \max\limits_{k=1}^m (T_{cr}^{l(Z_k)} -l(Z_k) +1)  \leq 
    \max\limits_{k=1}^m l(Z_k)  \left( \left\lfloor\frac{d}{\gamma}\right\rfloor -2 \right) + d\\
    & =\girth\left( \left\lfloor\frac{d}{\gamma}\right\rfloor -2 \right) + d
    =\kim(\gamma,\girth,d). 
\end{align*}
Taking this further, when $\frac{l(Z_k)}{\gamma} \leq  \left\lfloor\frac{d}{\gamma}\right\rfloor -1$, we obtain
\begin{align*}
     l(Z_k)  \left( \left\lfloor\frac{d}{\gamma}\right\rfloor -2 \right) + d & = \gamma \left( \frac{ l(Z_k)}{\gamma}  \left( \left\lfloor\frac{d}{\gamma}\right\rfloor - 2 \right) + \frac{d}{\gamma} \right) \\
    & \leq \gamma 
    \left(\left\lfloor\frac{d}{\gamma}\right\rfloor-1\right) \left(\left\lfloor\frac{d}{\gamma}\right\rfloor-2\right)  + (d\rem\gamma)+ \gamma\left\lfloor\frac{d}{\gamma}\right\rfloor \\
    & = \gamma \left(\left\lfloor\frac{d}{\gamma}\right\rfloor^{2} -3\left\lfloor\frac{d}{\gamma}\right\rfloor+ 2 \right) +\gamma\left\lfloor\frac{d}{\gamma}\right\rfloor + (d\rem \gamma) \\
    & = \gamma \wiel\left(\left\lfloor\frac{d}{\gamma}\right\rfloor\right) + (d \rem \gamma)=\sch(\gamma,d).
\end{align*}
Otherwise, in the case when $\frac{l(Z_k)}{\gamma}=\left\lfloor\frac{d}{\gamma}\right\rfloor$ for some $k$
we use Proposition~\ref{p:TcRn} to obtain  
\begin{equation*}
\begin{split}
T_{cr}^{\gamma\left\lfloor \frac{d}{\gamma}\right\rfloor} 
(A,Z_k)-l(Z_k)+1
& \leq \gamma\left(\left\lfloor \frac{d}{\gamma}\right\rfloor-1\right)^2+\gamma+d-1-\gamma
\left\lfloor\frac{d}{\gamma}\right\rfloor+1\\
&= \gamma \wiel\left(\left\lfloor\frac{d}{\gamma}\right\rfloor\right)+(d\rem\gamma) =\sch(\gamma,d). 
\end{split}
\end{equation*}
Thus treating these two cases yields the first bound in Theorem~\ref{th:KimSchwarzNacht} in the case 
$\left\lfloor\frac{d}{\gamma}\right\rfloor>1$. The remaining case $\left\lfloor\frac{d}{\gamma}\right\rfloor=1$ was considered in Proposition~\ref{p:n/gamma=1}.
\end{proof}

\subsection{The case of cycle threshold expansion}

In this section we obtain a new bound for the Cycle Threshold scheme using the bounds for the cycle removal threshold obtained previously. It will use the 
following bound on $T_1(A,B):$

\begin{equation}
\label{e:t1CT-ineq}
T_1(A,B)\leq \max \left\lbrace T_{cr}^{l(\cycle)}(\cycle) + 1 \mid \cycle \text{ cycle in } \subcrit\right\rbrace     
\end{equation}
Here $\digr$ is a completely reducible subgraph of $\digr(A)$.

\begin{proposition}[{\cite[Proposition 6.5]{wCSR}}]\label{p:TcrToT1CT}
When $B=B_{CT}$, bound~\eqref{e:t1CT-ineq} holds 
with $\digr=\subcrit^{ct}$.
\end{proposition}
Subgraph $\digr^{ct}$ was defined when we described Cycle Threshold scheme in Subsection~\ref{ss:CSR}. Note that this is a completely reducible subgraph of $\digr(A)$ that contains $\crit(A)$.

\begin{theorem} \label{th:T1CT}
If bound~\eqref{e:t1CT-ineq} holds, and in particular for $B=B_{\ct}(A)$, then
we also have the following bound:
\begin{equation*}
T_1(A,B) \le \gamma \left( \left\lfloor\frac{d}{\gamma} \right\rfloor -1 \right)^{2} +d +\gamma    
\end{equation*}
\end{theorem}

Before we prove this theorem, we first introduce the following lemma

\begin{lemma}\label{l:Tcrcomp}
Let $A\in\Rmax^{d\times d}$ and $\cycle$ be an elementary cycle of $\digr(A)$ of length $l(\cycle)$ and let $\digr(A)$ have cyclicity $\gamma$. Then either $l(Z)=\gamma\left\lfloor\frac{d}{\gamma}\right\rfloor $ or
\begin{equation} \label{eq:Tcrcomp}
    l(\cycle)\left(\left\lfloor\frac{d}{\gamma}\right\rfloor -1 \right) + d \leq \gamma \left(\left\lfloor\frac{d}{\gamma}\right\rfloor -1 \right)^2 +d +\gamma-1.
\end{equation}
\end{lemma}
\begin{proof}
As the length of any cycle is a multiple of $\gamma$, we see that $\gamma\left\lfloor\frac{d}{\gamma}\right\rfloor$ is the biggest possible length of a cycle of $\digr(A)$. As all other lengths of cycles are bounded by $l(\cycle)\leq\gamma\left\lfloor\frac{d}{\gamma}\right\rfloor - \gamma$, we substitute it into the left hand side of the inequality~\eqref{eq:Tcrcomp} to give
\begin{align*}
     l(\cycle)\left(\left\lfloor\frac{d}{\gamma}\right\rfloor -1 \right) + d & \leq  \left(\gamma\left\lfloor\frac{d}{\gamma}\right\rfloor - \gamma\right) \left(\left\lfloor\frac{d}{\gamma}\right\rfloor -1 \right) + d \\
     &\leq \gamma \left(\left\lfloor\frac{d}{\gamma}\right\rfloor -1 \right)^2 +d + \gamma-1,
\end{align*}
as required.
\end{proof}
\begin{proof}[Proof of Theorem~\ref{th:T1CT}]
We can split this proof into two distinct cases, the first is when there is a cycle of maximal length, which is $l(\cycle)=\gamma\left\lfloor\frac{d}{\gamma}\right\rfloor$, and the second is when every cycle has length that is smaller than the maximal possible length, i.e., $l(\cycle) < \gamma\left\lfloor\frac{d}{\gamma}\right\rfloor$.

For the first case we can use Proposition~\ref{p:TcRn} for maximal cycle length to give
\begin{equation}
\label{e:firstcase}
    T_{cr}^{\gamma\left\lfloor\frac{d}{\gamma}\right\rfloor}(A,Z) \le \gamma \left( \left\lfloor\frac{d}{\gamma}\right\rfloor -1 \right)^{2} +d +\gamma -1
\end{equation}

Turning to the second case, we can use Corollary~\ref{c:TcRLin}, which 
means that
\begin{align*}
    T_{cr}^{l(\cycle)}(A,Z) & \le l(\cycle) \left\lfloor\frac{d}{\gamma}\right\rfloor +d -l(\cycle) -1, \\
    T_{cr}^{l(\cycle)}(A,Z) + 1 & \le l(\cycle) \left\lfloor\frac{d}{\gamma}\right\rfloor +d -l(\cycle) 
       = l(\cycle) \left( \left\lfloor\frac{d}{\gamma}\right\rfloor -1 \right) +d.
\end{align*}
We can use Lemma~\ref{l:Tcrcomp} to bound this from above to get
\begin{equation}
\label{e:secondcase}
    T_{cr}^{l(Z)}(A,Z)+1 \le  l(\cycle)\left(\left\lfloor\frac{d}{\gamma}\right\rfloor -1 \right) + d \leq \gamma \left(\left\lfloor\frac{d}{\gamma}\right\rfloor -1 \right)^2 +d +\gamma.
\end{equation}

Using~\eqref{e:firstcase} and~\eqref{e:secondcase}
we obtain
\begin{equation*}
    T_1(A,B) \le  \max_{\cycle} \left\lbrace T_{cr}^{\gamma\left\lfloor\frac{d}{\gamma}\right\rfloor}(A,Z) +1 \right\rbrace \le \gamma \left( \left\lfloor\frac{d}{\gamma}\right\rfloor -1 \right)^{2} +d +\gamma.
\end{equation*}

\end{proof}

\section{Bounds for $T_2(A,B)$}

\label{s:T2AB}

In this section we develop new bounds for $T_2(A,B)$,
where $B$ is a {\em subordinate} to $A$, i.e., a matrix obtained 
from $A$ by setting some entries of $A$ to $-\infty$ 
and keeping all other entries the same as in $A$. In particular, $B_{\nacht}(A)$, $B_{\harg}(A)$ and $B_{\ct}(A)$ are 
subordinate matrices. 

These new bounds, presented in Theorem~\ref{th:T2TcrBounds},  are based on 1) an improvement of a bound of~\cite{wCSR} connecting $T_2(A,B)$ with $T_{cr}$ (which we replace with $T_{wr}$) and 2) new bounds on $T_{cr}$ and $T_{wr}$ obtained in Section~\ref{s:Thresh}.

\if{
a new type of bound for the CSR threshold which does not depend on the type of decomposition of $B$. It only requires that $B$ is subordinate to $A$ i.e. $B$ is extracted from $A$ by setting rows and/or columns to the max-plus zero ($-\infty$). We begin with a defintion.
}\fi

Let us recall the definition of $T_2(A,B)$, for a given irreducible $A\in\Rmax^{d\times d}$ and for a subordinate $B$ of $A$: by $T_2(A,B)$ 
we denote the smallest integer $T$ satisfying
\begin{equation}\label{eq:T2}
    \forall \; t\geq T, \; \; \; \lambda^{\otimes t}\otimes CS^{t}R[A] \geq B^{t}.
\end{equation}
In the paper~\cite{wCSR}, multiple bounds were developed for $T_{2}$ using bounds for the cycle removal threshold (from the same paper). We are going to improve the following 
bounds from \cite{wCSR}:
\begin{proposition}[{\cite[Theorem 4.5]{wCSR}}] \label{th:T2Bounds}
Let $A \; \in \; \Rmax^{d\times d}$ be irreducible and let $B$ be subordinate to $A$. Denote by $\cd_{B} = \cd(\digr(B))$ the length of the longest path in the associated digraph of $B$ and by $\hat{\sigma}$ the maximal cyclicity of the components of $\crit(A)$.
If $\sr(B) = -\infty$, then $T_{2}(A,B) \leq \cd_{B} +1 \leq n_{B}$. Otherwise we have the following bounds

\begin{equation} \label{eq:T2Bound1}
    T_2(A,B) \le \frac{(d^2-d+1)(\sr(A)-\min_{ij}a_{ij})+\cd_{B}(\max_{ij}b_{ij}-\sr(B))}{\sr(A)-\sr(B)}
\end{equation}
\begin{equation} 
    \label{eq:T2Bound2}
    T_2(A,B) \le \frac{(\hat{\sigma}(d-1)+d-1)(\sr(A)-\min_{ij}a_{ij})+\cd_{B}(\max_{ij}b_{ij}-\sr(B))}{\sr(A)-\sr(B)}
\end{equation}
\end{proposition}

To be able to use the cycle removal thresholds developed in this paper we need the following proposition. 

\begin{proposition} \label{p:T2Tcr}
Let $A$ be an irreducible matrix and $\tilde{A}=A-\sr(A)$, $\subcrit$ be a representing subgraph of $\subcrit^{c}(A)$ with s.c.c.'s $\subcrit_{1},\ldots,\subcrit_{m}$ and let $\sigma_l$ be the cyclicity of $\subcrit_{l}$. Let $B$ be subordinate to $A$ such that $\sr(B)\neq -\infty$. Then
\begin{equation} \label{eq:T2Tcr}
    T_2(A,B) \le \frac{\max_{i}(T_{wr}^{\sigma_i}(\Tilde{A},\subcrit_i))(\sr(A)-\min_{ij}a_{ij})+\cd_{B}(\max_{ij}b_{ij}-\sr(B)).}{\sr(A)-\sr(B)}.
\end{equation}
\end{proposition}
This proposition is inspired by \cite[Theorem 10.1]{wCSR} but it is different since we need to have the maximum of $T_{wr}$ over subgraphs $\subcrit_{l}$ in the bound. Also note that $T_{cr}$ is replaced with $T_{wr}$. Therefore it will require a proof.
\begin{proof}
Assume that $t$ is greater than the right hand side of \eqref{eq:T2Tcr}. We need to prove that 
\begin{equation} \label{eq:T2Cond}
    t\sr(A) \otimes (CS^{t}R[\Tilde{A}])_{ij} \geq t\sr(B) \otimes \tilde{B}^t_{ij}
\end{equation}
holds for all $i$,$j$, where $\Tilde{A}=A-\sr(A)$ 
and $\Tilde{B}=B-\sr(B)$.
Before we begin this, recall that $(CS^tR[\tilde{A}])_{ij}=(CS^tR[A])_{ij}$. 
Using \cite[Theorem 6.1]{wCSR} and \cite[Corollary 6.2]{wCSR} we have that
\begin{equation*}
    (CS^{t}R[\tilde{A}])_{ij} 
    = \max_{\nu=1,\ldots,m}\left(p\left(\mathcal{W}^{t,\sigma_{\nu}}\left(i \xrightarrow{\subcrit_{\nu}} j \right)\right)\right).
\end{equation*}
where $\sigma_{\nu}$ is
the cyclicity of $\subcrit_{\nu}$, and the weights are computed in $\digr(\Tilde{A})$. If $(CS^tR[\Tilde{A}])_{ij}$ is finite then one of the sets
$\mathcal{W}^{t,\sigma_{\nu}}\left(i \xrightarrow{\subcrit_{\nu}} j \right)$ is non-empty. Let it 
be non-empty for $\nu=\mu$ for some $\mu$, then we have:
$$ (CS^tR[\Tilde{A}])_{ij}\geq p\left(\mathcal{W}^{t,\sigma_{\mu}}\left(i \xrightarrow{\subcrit_{\mu}} j \right)\right)\geq T_{wr}^{\sigma_{\mu}}(\tilde{A},\subcrit_{\mu})\min_{k,l} \tilde{a}_{kl},
$$
as $\mathcal{W}^{t,\sigma_{\mu}}\left(i \xrightarrow{\subcrit_{\mu}} j \right)$ contains a 
walk whose length does not exceed $T_{wr}^{\sigma_{\mu}}(\subcrit_{\mu})$ and as $\min_{k,l} \Tilde{a}_{kl}$ 
is non-positive. 
We further  obtain that 
\begin{equation}
\label{e:CSRbelow}
(CS^tR[\Tilde{A}])_{ij}\geq \min_{\nu}\left(T_{wr}^{\sigma_{\nu}}(\Tilde{A},\subcrit_{\nu})\min_{kl} \Tilde{a}_{kl}\right).
\end{equation}

By \cite[Lemma 10.2]{wCSR} if the entry $(CS^tR[A])_{ij}$ is not finite then neither is $\tilde{B}^{t}_{ij}$ and there is nothing to prove, so, we assume that $(CS^tR[A])_{ij}=(CS^tR[\Tilde{A}])_{ij}$) is finite.  Passing to 
$A=\sr(A)\otimes \Tilde{A}$, we then use~\eqref{e:CSRbelow}
to argue that the inequality
\begin{equation}
    t \sr(A) + \min_{\nu}\left(T_{wr}^{\sigma_\nu}(\tilde{A},\subcrit_{\nu})( \min_{kl}a_{kl} -\sr(A)) \right) \geq t\sr(B) + \cd(\digr{(B)})( \max_{kl} b_{kl} - \sr(B)) 
\end{equation}    
guarantees~\eqref{eq:T2Cond}. Rearranging the last inequality we obtain
 \begin{equation}
 \begin{split}
    t(\sr(A)-\sr(B)) \geq   &\max_{\nu}\left(T_{wr}^{\sigma_\nu}(\tilde{A},\subcrit_{\nu})(\sr(A) - \min_{kl}a_{kl}) \right)\\ 
    &+ \cd\left( \digr{(B)} \right)( \max_{kl} b_{kl} - \sr(B)),
\end{split}
\end{equation}

Since $\left(\sr(A) - \min_{kl}a_{kl} \right)$ does not depend on $\nu$ and $\sr(A) \ge \sr(B)$,
dividing this inequality by $\sr(A)-\sr(B)$ does not change its sign and makes its right-hand side identical with that of~\eqref{eq:T2Tcr}. Therefore any $t$ greater than~\eqref{eq:T2Tcr} will satisfy~\eqref{eq:T2Cond} as well, thus completing the proof.
\end{proof}
Using this proposition along with Corollary~\ref{c:TcRLin} and Proposition~\ref{p:TcRn} we can prove new bounds for $T_2$.

\begin{theorem} \label{th:T2TcrBounds}
Let $A \in \Rmax^{d\times d}$ be irreducible
with cyclicity $\gamma$ and factor rank~$r$ and let $B$ be subordinate to $A$ such that $\sr(B)\neq -\infty$. Then the following bounds on $T_{2}(A,B)$ hold.

\begin{align} 
\label{eq:T2TcrBoundSchwartz}
T_{2}(A,B) & \le \frac{\left(\gamma\wiel\left(\left\lfloor\frac{d}{\gamma}\right\rfloor\right) +d -1\right)(\sr(A)-\min_{ij}a_{ij})+\cd_{B}(\max_{ij}b_{ij}-\sr(B))}{\sr(A)-\sr(B)},
\\
\label{eq:T2TcrBoundDM}
     T_{2}(A,B) & \le \frac{\left(\hat{\sigma}\left(\left\lfloor\frac{d}{\gamma}\right\rfloor -1 \right)+d -1\right)(\sr(A)-\min_{ij}a_{ij})+\cd_{B}(\max_{ij}b_{ij}-\sr(B))}{\sr(A)-\sr(B)}, 
\\
\label{eq:T2TcrBoundrankSchwartz}
     T_{2}(A,B) & \le \frac{\left(\gamma\wiel\left(\left\lfloor\frac{r}{\gamma}\right\rfloor\right)+r +\gamma\right)(\sr(A)-\min_{ij}a_{ij})+\cd_{B}(\max_{ij}b_{ij}-\sr(B))}{\sr(A)-\sr(B)},
\\
\label{eq:T2TcrBoundrankDM}
     T_{2}(A,B) & \le \frac{\left(\hat{\sigma}\left(\left\lfloor\frac{r}{\gamma}\right\rfloor -1 \right)+r +\gamma\right)(\sr(A)-\min_{ij}a_{ij})+\cd_{B}(\max_{ij}b_{ij}-\sr(B))}{\sr(A)-\sr(B)},
\end{align}

where $\hat{\sigma}$ is the greatest cyclicity of the 
s.c.c. of $\crit(A)$.
\end{theorem}
\begin{proof}
For the first bound we recall that the length of each cycle does
not exceed $\gamma\left\lfloor\frac{d}{\gamma}\right\rfloor$, and the second largest length does not exceed $\gamma\left(\left\lfloor\frac{d}{\gamma}\right\rfloor-1\right).$
If the $\nu$'th s.c.c. of $\crit(A)$ has only cycles of the maximal length $\gamma\left\lfloor\frac{d}{\gamma}\right\rfloor$ 
then denoting one of such cycles by $Z_{\nu}$ and using Proposition~\ref{p:TcRn} 
we have
\begin{equation*}
   T_{cr}^{l(Z_{\nu})}(A,Z_{\nu}) \le \gamma \left( \left\lfloor\frac{d}{\gamma}\right\rfloor -1 \right)^{2} +d +\gamma -1.
\end{equation*}
If it has a cycle $Z_{\nu}$ with smaller length, then using Corollary~\ref{c:TcRLin}
we obtain
\begin{equation*}
    T_{cr}^{l(Z_{\nu})}(A,Z_{\nu}) \le l(\cycle_{\nu})\left\lfloor\frac{d}{\gamma}\right\rfloor -l(\cycle_{\nu})+d-1.
\end{equation*}
We can bound this from above using Lemma~\ref{l:Tcrcomp} to get again 
that
\begin{equation*}
    T_{cr}^{l(Z_{\nu})}(A,Z_{\nu})\leq \gamma \left( \left\lfloor\frac{d}{\gamma}\right\rfloor -1 \right)^{2} +d +\gamma -1.
\end{equation*}
Using that $T_{wr}\leq T_{cr}$ we can substitute the above bound into Proposition~\ref{p:T2Tcr}, where we set $\subcrit_{\nu}=Z_{\nu}$ for each $\nu$.  
Thus we get
the first bound~\eqref{eq:T2TcrBoundSchwartz}.

For the second bound, we set $\subcrit_{\nu}$ to be the s.c.c.'s of $\crit(A)$. Using Proposition~\ref{p:TcRLinGen} we obtain
\begin{equation*}
    T_{cr}^{\sigma_{\nu}}(A,\subcrit_{\nu}) \le 
    \sigma_{\nu}\left\lfloor\frac{d}{\gamma}\right\rfloor-\sigma_{\nu}+d-1, 
\end{equation*}
where $\sigma_{\nu}$ is the cyclicity of $\subcrit_{\nu}$.
Substituting this into Proposition~\ref{p:T2Tcr} we get the second bound~\eqref{eq:T2TcrBoundDM}.

For the third bound, we assume without loss of generality that $\sr(A)=0$. In each component of $\crit(A)$ we select a cycle of minimal length. Let $Z_{\nu}$ be such a cycle in $\nu$th component. By Lemma~\ref{l:girth-related}, for any $Z_{\nu}$ there is a related cycle $\check{Z}_{\nu}$ in the related s.c.c. of $\crit(\check{A})$ with the same length, and this length is bounded from above by 
$\gamma\left\lfloor \frac{r}{\gamma}\right\rfloor$ (since $r$ is the dimension of $\check{A}$). Proposition~\ref{p:Twr} (ii) then gives the following bound:
\begin{equation*}
T_{wr}^{l(Z_{\nu})}(A,Z_{\nu})\leq l(Z_{\nu})\left\lfloor\frac{r}{\gamma}\right\rfloor+r-l(Z_{\nu})+\gamma.    
\end{equation*}
If $l(Z_{\nu})\leq \gamma\left(\left\lfloor \frac{r}{\gamma}\right\rfloor-1\right)$, then the right-hand side of this bound is bounded by
$$
\gamma\wiel\left(\left\lfloor \frac{r}{\gamma}\right\rfloor\right) +r+\gamma,
$$
by an argument similar to Lemma~\ref{l:Tcrcomp}. If 
$l(Z_{\nu})=\gamma\left\lfloor \frac{r}{\gamma}\right\rfloor,$ then 
\begin{equation*}
T_{wr}^{l(Z_{\nu})}(A,Z_{\nu})\leq \gamma\wiel\left(\left\lfloor \frac{r}{\gamma}\right\rfloor\right) +r+\gamma
\end{equation*}
from Proposition~\ref{p:Twr} (iv). We then substitute this bound in Proposition~\ref{p:T2Tcr}, where we set $\subcrit_{\nu}=Z_{\nu}$ for each $\nu$,
and obtain the third bound.

For the fourth bound, using Proposition~\ref{p:Twr} (ii), we obtain 
\begin{equation*}
    T_{wr}^{\sigma_{\nu}}(A,\subcrit_{\nu}) \le 
    \sigma_{\nu}\left\lfloor\frac{r}{\gamma}\right\rfloor+r-\sigma_{\nu}+\gamma, 
\end{equation*}
where $\subcrit_{\nu}$ is a s.c.c. of $\crit(A)$ and $\sigma_{\nu}$ is the cyclicity of this component. Here we also use that 
$\sigma_{\nu}$ is a multiple of $\gamma$ (hence it can be taken for $l$)  and that $\circumf(\check\subcrit_{\nu})$ (circumference of the s.c.c. of $\crit(\check{A})$ which is related to 
$\subcrit_{\nu}$) is not smaller than the cyclicity of $\check\subcrit_{\nu}$ (equal to $\sigma_{\nu}$ by Lemma~\ref{l:crit-related}). 
\end{proof}

With these bounds it remains to check that they are better than the previous ones. Obviously, \eqref{eq:T2TcrBoundDM} is better than 
\eqref{eq:T2Bound2}, and it remains to compare \eqref{eq:T2TcrBoundSchwartz}
with \eqref{eq:T2Bound1}. This is achieved in the following
\begin{proposition}\label{p:T2Comp}
For any irreducible matrix $A\in \Rmax^{d\times d}$ with subordinate matrix $B$, the bound~\eqref{eq:T2TcrBoundSchwartz} is smaller than the
bound~\eqref{eq:T2Bound1}.
\end{proposition}
\begin{proof}
Upon comparing the two bounds the inequality simplifies down to trying to prove that
\begin{equation*}
    \gamma\left(\left\lfloor\frac{d}{\gamma}\right\rfloor-1\right)^{2} +\gamma +d -1 \leq d^2 -d +1,
\end{equation*}
which is the same as
\begin{equation*}
 \gamma\left(\left\lfloor\frac{d}{\gamma}\right\rfloor-1\right)^{2} +\gamma  \leq d^2 -2d +2.
\end{equation*}

We will prove the slightly stronger
\begin{equation*}
 \gamma\left(\frac{d}{\gamma}-1\right)^{2} +\gamma  \leq d^2 -2d +2,
\end{equation*}
which is the same as
\begin{equation*}
 \frac{d^2}{\gamma} +2\gamma  \leq d^2 +2,
\end{equation*}
and as
\begin{equation*}
 2\gamma\left(1-\frac{1}{\gamma}\right)  \leq d^2\left(1-\frac{1}{\gamma}\right),
\end{equation*}
and finally as
$$2\gamma\le d^2,$$
which holds whenever $d\ge 2$.
The case $d=1$ being trivial, the proposition is proved.
\end{proof}

\providecommand{\bysame}{\leavevmode\hbox
to3em{\hrulefill}\thinspace}
\providecommand{\MR}{\relax\ifhmode\unskip\space\fi MR }
\providecommand{\MRhref}[2]{%
  \href{http://www.ams.org/mathscinet-getitem?mr=#1}{#2}
} \providecommand{\href}[2]{#2}


\begin{thebibliography}{10}
\bibitem{AZ:01}
M. Aigner and G. M. Ziegler, 
\newblock {\em Proofs from {T}he {B}ook},
\newblock   Springer, Heidelberg, 2001, Second edition.

\bibitem{AGG-09} M. Akian, S. Gaubert, and A. Guterman. 
Linear independence over tropical semirings and beyond. 
In: G.L. Litvinov and S.N. Sergeev (eds.) {\em Tropical and Idempotent 
Mathematics}, Contemporary Mathematics 495, AMS, Providence, 2009, 
pp. 1--38. 


\bibitem{AGW-05} M. Akian, S. Gaubert, and C. Walsh.
    Discrete max-plus spectral theory. In: G. L. Litvinov and V. P. Maslov
(eds.) {\em Idempotent Mathematics and Mathematical Physics}, Contemporary
Mathematics 377, AMS, Providence, 2005, pp.~53--77.

\bibitem{BCOQ} F.L. Baccelli, G. Cohen, G. J. Olsder, and
J.P. Quadrat. {\em Synchronization and Linearity: An Algebra for
Discrete Event Systems}. John Wiley \& Sons, Hoboken, 1992.
\bibitem{BG-00}
A. Bouillard and B. Gaujal.
Coupling time of a (max,plus) matrix. {\em Proceedings of the Workshop on Max-Plus Algebra at the 1st
                IFAC Symposium on System Structure and Control}, pp.~{335--400}. Elsevier, Amsterdam, 2001.
\bibitem{BR} R. A. Brualdi and H. J. Ryser.
    {\em Combinatorial Matrix Theory}.
    Cambridge University Press, Cambridge, 1991.

\bibitem{But:10} P. Butkovi\v{c}, {\em Max-linear Systems: Theory and Algorithms}.
Springer, London, 2010.


\bibitem{CBFN-12} B. Charron-Bost, M. F{\"u}gger, and T.
Nowak.
    New transience bounds for long walks, 2012.
    \ \ {\tt arXiv:1209.3342 [cs.DM]}




\bibitem{CDQV-83} G. Cohen, D. Dubois, J.P. Quadrat, and M.
Viot. Analyse du comportement p\'eriodique de syst\`emes de production par la
th\'eorie des dio\"ides. INRIA Research Report 191, INRIA, Le Chesnay, 1983.

\bibitem{DSS-05} M. Develin, F. Santos, and B. Sturmfels. 
On the rank of a tropical matrix. In J.E. Goodman, J. Pach and E. Welzl (Eds.), 
{\em Combinatorial and computational 
geometry,}  MSRI Publications 5, pp. 213--242. Cambridge Univ. Press, 
Cambridge, 2005.


\bibitem{DM-62}
A.~L.~Dulmage and N.~S.~Mendelsohn.
\newblock Gaps in the exponent set of primitive matrices.
\newblock {\em Illinois Journal of Mathematics}, 8(4):642--656, 1964.

\bibitem{GKP-95}
D. A. Gregory, S. J. Kirkland, and N. J. Pullman.
\newblock A bound on the exponent of a primitive matrix using {B}oolean rank.
\newblock {\em Linear Algebra and Its Applications}, 217:101--116, 1995.

\bibitem{HA-99}  M. Hartmann and C. Arguelles. Transience bounds for
long walks. {\em Mathematics of Operations Research\/}~24(2):414--439, 1999.

\bibitem{Kim-79} K. H. Kim. An extension of the Dulmage-Mendelsohn
theorem. {\em Linear Algebra and Its Applications}, 27:187-197,
1979.



\bibitem{CritCol}
G.~Merlet, T.~Nowak, H.~Schneider and S.~Sergeev.
\newblock Generlizations of bounds on the index of convergence to weighted digraphs.
\newblock {\em Discrete Applied Mathematics}, 178:121--134, 2014.

\bibitem{wCSR}
G.~Merlet, T.~Nowak and S.~Sergeev.
\newblock Weak CSR expansions and transience bounds in max-plus algebra.
\newblock {\em Linear Algebra and its Applications}, 461:163--199, 2014.

\bibitem{ReachBnds}
G.~Merlet, T.~Nowak and S.~Sergeev.
\newblock On the tightness of bounds for transients of weak CSR expansions and periodicity transients of 
critical rows and columns of tropical matrix powers. ArXiv preprint: \href{1705.04104}{https://arxiv.org/pdf/1705.04104.pdf}


\bibitem{Nacht} K. Nachtigall. Powers of matrices over an extremal algebra
with applications to periodic graphs. {\em Mathematical Methods of Operations
Research\/}~46(1):87--101, 1997.


\bibitem{S-85}
J.-Y.~Shao.
\newblock On the exponent of a primitive digraph.
\newblock {\em Linear Algebra and its Applications}, 64:21--31 ,1985.




\bibitem{SchSch}  H. Schneider and M. H. Schneider. Max-balancing
weighted directed graphs and matrix scaling. {\em Mathematics of Operations
Research\/}~16(1):208--222, 1991.

 
\bibitem{Sch-70}
\v{S}. Schwarz.
\newblock On a sharp estimation in the theory of
binary relations on a finite set.
\newblock {\em Czechoslovak Mathematical Journal},
20:703--714, 1970.

\bibitem{Ser-09}
S. Sergeev. 
\newblock Max algebraic powers of irreducible
matrices in
  the periodic regime: An application of cyclic classes. 
\newblock {\em Linear Algebra and Its Applications\/}, 431(6):1325--1339, 2009.


\bibitem{SerSch}
 S. Sergeev and H. Schneider.
\newblock CSR expansions of
matrix powers in max algebra, 
\newblock {\em Transactions of the AMS\/},
364:5969--5994, 2012.

\bibitem{SyK:03} Gerardo Soto y Koelemeijer. {\em On the Behaviour of Classes
of Min-Max-Plus Systems}. PhD Thesis, TU Delft, 2003.

\bibitem{Wie-50} H. Wielandt. Unzerlegbare, nicht negative Matrizen.
 {\em Mathematische Zeitschrift}~ 52(1):642--645, 1950.


\end{thebibliography}
\end{document}